\documentclass[a4paper,10pt]{amsart}
\usepackage{amsmath,amsthm,amssymb,amsfonts,enumerate,color}
\usepackage[pdftex]{graphicx}

\newcommand{\R}{\mathbb{R}}

\newcommand{\N}{\mathbb{N}}

\newcommand{\dive}{\operatorname{div}}

\newcommand{\x}{\mathbf{x}}
\newcommand{\y}{\mathbf{y}}
\newcommand{\z}{\mathbf{z}}
\newcommand{\X}{\mathbf{X}}

\newtheorem{thm}{Theorem}[section]

\newtheorem{lem}[thm]{Lemma}

\theoremstyle{definition}
\newtheorem{defn}[thm]{Definition}
\newtheorem{rem}[thm]{Remark}

\numberwithin{equation}{section}

\allowdisplaybreaks

\author[L. Ping]{Li Ping}
    \address{School of Mathematics and Statistics \\
    Wuhan University\\
    430072 Wuhan, China and\\
    Departamento de Matem\'aticas\\
    Universidad Aut\'onoma de Madrid\\
    28049 Madrid, Spain}
\email{marrisa0708@163.com}

\author[P. R. Stinga]{Pablo Ra\'ul Stinga}
    \address{Department of Mathematics\\
    Iowa State University\\
    396 Carver Hall, Ames, IA 50011, United States of America}
    \email{stinga@iastate.edu}

\author[J. L. Torrea]{Jos\'e L. Torrea}
    \address{Departamento de Matem\'aticas\\
    Universidad Aut\'onoma de Madrid\\
    28049 Madrid, Spain}
\email{joseluis.torrea@uam.es}



\begin{document}

\title[Weighted mixed-norm Sobolev estimates for parabolic equations]{On weighted mixed-norm
Sobolev estimates \\ for some basic parabolic equations}

\begin{abstract}
Novel global weighted parabolic Sobolev estimates, weighted mixed-norm
estimates and a.e. convergence results of
singular integrals for evolution equations are obtained. Our results include the classical heat equation,
the harmonic oscillator evolution equation
$$\partial_tu=\Delta u-|x|^2u+f,$$
and their corresponding Cauchy problems. We also show weighted mixed-norm estimates
for solutions to degenerate parabolic extension problems arising in connection with the fractional 
space-time nonlocal equations $(\partial_t-\Delta)^su=f$ and $(\partial_t-\Delta+|x|^2)^su=f$, for $0<s<1$.
\end{abstract}

\maketitle

\section{Introduction}

The theory of elliptic PDEs received an unexpected major impulse in 1952 with the fundamental work
of A. P. Calder\'on and A. Zygmund \cite{Calderon-Zygmund} on Sobolev \textit{a priori} $W^{2,p}$ estimates of solutions.
Calder\'on and Zygmund exploited their ideas to cover a large class of PDEs.
In particular, if $P(D)$ is a linear homogeneous partial differential operator of order $m$ with smooth coefficients
then $P=H\Lambda^m$ where $\Lambda=(-\Delta)^{1/2}$, the square root of the Laplacian,
and $H$ is a singular integral operator.
Regarding this property they asserted in \cite{Calderon-Zygmund 2}:
``This fact seems to call for a closer study of the properties of singular integral operators
in their connection with the operator $\Lambda$''.

After the appearance of \cite{Calderon-Zygmund}
some attempts were made to obtain Sobolev estimates for solutions of
parabolic PDEs. Probably the most known work is the 1964 paper
by B. F. Jones \cite{Jones}. Jones studies parabolic problems of
the form $\partial_tu=(-1)^{m/2}P(D)u+f$, for $t>0$, $x\in\R^n$, with $u(0,x)=0$,
where $P(\xi)= P(\xi_1,\dots,\xi_n)$ is a homogeneous polynomial
of even degree $m$, such that $P(\xi)$ has negative real part for real $\xi$.
The parabolic Calder\'on--Zygmund estimate in this case says that if
$f\in L^p(\R^{n+1}_+)$, where $\R^{n+1}_+:=(0,\infty)\times\R^{n}$, for $1<p<\infty$,
then $\partial_tu,D^2u\in L^p(\R^{n+1}_+)$.
In 1966, E. B. Fabes extended in \cite{Fabes} the results to variable kernel operators
and provided new applications. That same year, E. B. Fabes and C. Sadosky
proved an almost everywhere convergence result of second derivatives $D^2u$
when $f\in L^p(\R^{n+1}_+)$ and $1<p<\infty$, see \cite{Fabes-Sadosky}.

Apart from the series of 1960's papers mentioned above, there have not been many
investigations on the parabolic Calder\'on--Zygmund theory during the last century.
In general, in specific books like \cite{Krylov book, Ames} there are only a few comments or small related chapters.
This is a surprisingly big difference with respect to the case of elliptic PDEs.
Special mention deserve the early 2000's papers by N. V. Krylov \cite{Krylov0, Krylov} (see also references therein)
where mixed norm estimates $L^q_t(W_x^{2,p})$ and $L^p_t(C_x^{2,\alpha})$ for parabolic equations were obtained.
It has brought again some of the primitive ideas of Calder\'on and Zygmund to the present times.

In this paper we aim to show weighted mixed-norm $L^q_t(L^p_x)$ and $L^p_{t,x}$ estimates,
weighted mixed weak-type estimates and a.e.~convergence results
of singular integrals for the following parabolic equations:
the heat equation
\begin{equation}\label{EQ1}
\partial_tu=\Delta u+f,\quad\hbox{in}~\R^{n+1},
\end{equation}
the harmonic oscillator evolution equation
\begin{equation}\label{para}
\partial_tu=\Delta u-|x|^2u+f,\quad\hbox{in}~\R^{n+1},
\end{equation}
and their corresponding Cauchy problems in $\R^{n+1}_+$.
We also prove similar estimates for some degenerate parabolic extension equations
connected with the fractional space-time nonlocal equations $(\partial_t-\Delta)^su=f$ and $(\partial_t-\Delta+|x|^2)^su=f$
in $\R^{n+1}$, see \eqref{eq:extension Delta} and \eqref{eq:extension H}. The latter
are of particular interest in regularity theory of
space-time fractional nonlocal PDEs, see \cite{Stinga-Torrea-s} and references therein.

The weights appearing in our first two main statements are the usual Muckenhoupt weights on $\R$ and $\R^n$.
We refer the reader to the book by J. Duoandikoetxea \cite[Chapter~7]{Duo} for definition and properties of the $A_p$ classes
and to Section \ref{pre} for the necessary notation. Here is our first main result.

\begin{thm}[Mixed-norm Sobolev estimates with weights]\label{mixriesz}
Let $f\in L^q(\R,\nu;L^p(\R^{n},\omega))$ for some $1\leq p,q<\infty$, where
$\nu\in A_q(\R)$ and $\omega \in A_p(\R^n)$. Let $u$ be either a solution to the
heat equation \eqref{EQ1} or to the harmonic oscillator evolution equation \eqref{para} in $\R^{n+1}$.
If $1<p,q<\infty$ then
$$\|\partial_{ij}u\|_{L^q(\R,\nu;L^p(\R^{n},\omega))}+\|\partial_tu\|_{L^q(\R,\nu;L^p(\R^{n},\omega))}\le
C_{n,p,q,\nu,\omega}\|f\|_{L^q(\R,\nu;L^p(\R^{n},\omega))}.$$
If $q=1$ and $1 < p<\infty$ then a weak-type estimate holds: for any $\lambda>0$, 
$$\nu\big( \{t \in\R:\|\partial_{ij}u(t,\cdot)\|_{L^p(\R^{n},\omega)}+\|\partial_{t}u(t,\cdot)\|_{L^p(\R^{n},\omega)}>\lambda\}\big)\le\frac{C_{n,p,\nu,\omega}}{\lambda}\|f\|_{L^1(\R,\nu;L^p(\R^{n},\omega))}.$$
\end{thm}

The strong estimate in Theorem \ref{mixriesz} for $D^2u$ and $\partial_t u$ in the unweighted case 
$\nu=\omega=1$ and when $u$ is a solution to the heat equation is already contained in a work 
by Krylov, see \cite{Krylov0} and references therein.
The rest of the estimates are completely new. In particular, we obtain the endpoint case $q=1$.
Notice also that the lower order coefficient in \eqref{para} is unbounded on $\R^{n+1}$.
In \cite{Hieber}, R. Haller-Dintelmann, H. Heck and M. Hieber 
proved $L^q([0,\infty);(L^p(\R^n,w))^N)$ estimates for solutions of non-divergence form
$N\times N$ parabolic systems of order $m$ on $\R^n$ with top-order coefficients
of class $VMO\cap L^\infty$, where
$1<p,q<\infty$ and $w\in A_p(\R^n)$. Observe that in such result for parabolic
systems with time independent variable coefficients
the mixed-norm estimate does not include weights with respect to the time variable nor the endpoint $q=1$.

The previous works on the parabolic Calder\'on--Zygmund theory
did not deal at any moment with the square root operator $(\partial_t-\Delta)^{1/2}$.
Recall that in the elliptic Calder\'on--Zygmund calculus the operator $\Lambda=(-\Delta)^{1/2}$
played a key role \cite{Calderon-Zygmund, Calderon-Zygmund 2}.
Very recently in \cite{Stinga-Torrea-s} a quite deep and complete analysis of
solutions to the fractional nonlocal equation $(\partial_t-\Delta)^su=f$ on $\R^{n+1}$, for $0<s<1$, was performed.
The fractional powers of the heat operator $(\partial_t-\Delta)^s$ can be characterized by
a degenerate parabolic extension problem in one more dimension.
Indeed, let $u=u(t,x)$ be a (smooth) bounded function on $\R^{n+1}$.
If $U=U(t,x,y)$ is a solution to the degenerate parabolic equation
\begin{equation}\label{eq:extension Delta}
\begin{cases}
\partial_tU=y^{-(1-2s)}\dive_{x,y}(y^{1-2s}\nabla_{x,y}U),&\hbox{for}~(t,x)\in\R^{n+1},~y>0,\\
U(t,x,0)=u(t,x),&\hbox{for}~(t,x)\in\R^{n+1},
\end{cases}
\end{equation}
then, for every $(t,x)\in\R^{n+1}$,
$$-\lim_{y\to0^+}y^{1-2s}U_y(t,x,y)=c_s(\partial_t-\Delta)^su(t,x),$$
where $c_s>0$ is an explicit constant.
Moreover, $U$ is given by the \textit{Poisson formula}
\begin{equation}\label{eq:Poisson operator Laplaciano}
U(t,x,y)\equiv P_{\Delta,y}^su(t,x)=\frac{y^{2s}}{4^s\Gamma(s)} \int_0^\infty e^{-y^2/(4\tau)}
e^{-\tau(\partial_t-\Delta)}u(t,x)\,\frac{d\tau}{\tau^{1+s}}.
\end{equation}
For all these details see \cite{Stinga-Torrea-s}.
Observe that when $s=1/2$ the operator $P_{\Delta,y}^{1/2}u$ can be thought as a subordinated
Poisson semigroup parallel to the one arising in elliptic PDEs \cite{Stein,Stempak-Torrea,Thangavelu}.
In a similar fashion, we define the Poisson operator related to the fractional power $(\partial_t-\Delta+|x|^2)^s$ as
\begin{equation}\label{eq:Poisson operator Hermite}
V(t,x,y)\equiv P^{s}_{H,y}u(t,x)=
\frac{y^{2s}}{4^s\Gamma(s)}\int_0^\infty e^{-y^2/(4\tau)}e^{-\tau(\partial_t-\Delta+|x|^2)}u(t,x)
\,\frac{d\tau}{\tau^{1+s}}.
\end{equation}
It can be checked that $V$ solves the degenerate parabolic equation with unbounded lower order term
\begin{equation}\label{eq:extension H}
\begin{cases}
\partial_tV=y^{-(1-2s)}\dive_{x,y}(y^{1-2s}\nabla_{x,y}V)+|x|^2V,&\hbox{for}~(t,x)\in\R^{n+1},~y>0,\\
V(t,x,0)=u(t,x),&\hbox{for}~(t,x)\in\R^{n+1},
\end{cases}
\end{equation}
and that
$$-\lim_{y\to0^+}y^{1-2s}V_y(t,x,y)=c_s(\partial_t-\Delta+|x|^2)^su(t,x).$$
To present our second main novel result, let us
denote by $\mathcal{P}^s_yu(t,x)$, $y>0$, any of the Poisson operators \eqref{eq:Poisson operator Laplaciano}
or \eqref{eq:Poisson operator Hermite}. Define the maximal operators as
$$\mathcal{P}^{s,\ast}u(t,x):=\sup_{y>0}|\mathcal{P}_y^su(t,x)|,\quad\hbox{for}~(t,x)\in\R^{n+1}.$$
As it is well known, these maximal operators are important to understand convergence of
the solutions $U$ and $V$ to the initial data $u$ or, in an equivalent way as evidenced
by the extension problems \eqref{eq:extension Delta} and \eqref{eq:extension H}, to solve the fractional
space-time nonlocal equations above.

\begin{thm}[Mixed-norm estimates with weights for extensions]\label{Poisson}
Let $u\in L^q(\R,\nu;L^p(\R^n,\omega))$ for $1\leq p,q<\infty$, where
$\nu\in A_q(\R)$ and $\omega\in A_p(\R^n)$.
If $1<p, q<\infty$ then 
$$\|\mathcal{P}^{s,\ast}u\|_{L^q(\R,\nu;L^p(\mathbb{R}^n,\omega))}\le
C_{n,p,q,\nu,\omega}\|u\|_{L^q(\R,\nu;L^p(\mathbb{R}^n,\omega))}.$$
If $q=1$ and $1< p <\infty$ then a weak-type estimate holds: for any $\lambda>0$, 
$$\nu\big(\{t\in\R:\|\mathcal{P}^{s,\ast}u(t,\cdot)\|_{L^p(\mathbb{R}^n,\omega)}>\lambda\}\big)
\le\frac{C_{n,p,\nu,\omega}}{\lambda}\|u\|_{L^1(\R,\nu;L^p(\R^{n},\omega))}.$$
\end{thm}

The last two main new results of this paper regard weighted estimates in parabolic Sobolev spaces
and a.e.~convergence of principal values for singular integrals.
Recall that the natural geometric setting for uniformly parabolic equations is given by the space
$\R^{n+1}$ endowed with the parabolic distance \eqref{eq:parabolic distance} and the Lebesgue measure.
These ingredients form a so-called \textit{space of homogeneous type}. The ``cubes'' in this geometry
are given by the parabolic distance. Therefore the class of Muckenhoupt weights defined in this setting,
and which we denote by $A_p^*(\R^{n+1})$, form the suited class for the non-mixed-weighted norm scenario.
Observe that this class is different from the usual $A_p(\R^{n+1})$ class. Moreover,
the next results do not follow from our previous Theorem \ref{mixriesz}. Here we present the estimates
for the harmonic oscillator evolution equation \eqref{para}, the case of the usual heat equation is contained
in Subsections \ref{subsection:1} and \ref{subsection:2}. The first result is for
the global equation in \eqref{para}.

\begin{thm}[Equation on the whole space]\label{fundamental1}
Let $ \mathcal{W}_\tau(x,y)$, $x,y\in\R^n$, $\tau>0$,
be the Mehler kernel of the heat semigroup generated by the harmonic oscillator
$H=-\Delta+|x|^2$ on $\R^n$, see \cite{Thangavelu} and \eqref{Hheat}.
\begin{itemize}
\item[({\bf A})]{\bf Classical solvability.}
Let $f=f(t,x)$ be a bounded function on
$\mathbb{R}^{n+1}$ with compact support. For every $(t,x) \in \mathbb{R}^{n+1} $ the integral  
$$u(t,x) =   \int_0^\infty \int_{\mathbb{R}^n} \mathcal{W}_\tau(x,y) f(t-\tau,y) \,dy\, d\tau,$$
is well defined. Moreover, if $f$  is a  $C^2$ function then $u$ is a classical solution to \eqref{para}.
In this case the following pointwise limits hold:
 \begin{equation}\label{equation x}
\partial_{ij}u(t,x)=\lim_{\varepsilon\rightarrow 0}\iint_{\Omega_\varepsilon(x)}\partial_{y_iy_j}\mathcal{W}_\tau(x,y)
f(t-\tau,y)\,dy\,d\tau-A_nf(t,x)\delta_{ij},
\end{equation}
and
\begin{equation}\label{equation t}
\partial_{t}u(t,x) =\lim_{\varepsilon \rightarrow 0}\iint_{\Omega_\varepsilon(x)}
\partial_{\tau}\mathcal{W}_\tau(x,y)f(t-\tau,y)\,dy\, d\tau+ B_n f(t,x).
\end{equation}
Here $A_n$ and $B_n$ are the explicit constants given in \eqref{eq:AyB} and, for $\varepsilon>0$,
$$\Omega_\varepsilon(x) = \{(t,y): {\rm max}( |t|^{1/2}, |x-y|) >\varepsilon\}.$$
\item[({\bf B})] {\bf Weighted parabolic Sobolev estimates.}   
Let $f\in L^p(\R^{n+1},w)$, for $1\le p<\infty$, where $w$ belongs to the parabolic
Muckenhoupt class $A_p^*(\R^{n+1})$ mentioned above. Then the limits
\eqref{equation x} and \eqref{equation t} exist for a.e.~$(t,x)\in\R^{n+1}$.
Moreover, the following weighted Sobolev a priori estimates hold: when $1<p<\infty$,
$$\|\partial_{ij}u\|_{L^p(\R^{n+1},w)}+\|\partial_tu\|_{L^p(\R^{n+1},w)}\leq C_{n,p,w}\|f\|_{L^p(\R^{n+1},w)}, $$
and, when $p=1$, for any $\lambda>0$,
$$w\big(\{(t,x)\in\R^{n+1}:|\partial_{ij}u|+|\partial_tu|>\lambda\}\big)
\le\frac{C_{n,w}}{\lambda}\|f\|_{L^1(\R^{n+1}, w)}.$$
\end{itemize}
\end{thm}

We point out that the global results we prove here for \eqref{para}, even in the unweighted case, are not covered by the
standard theory of parabolic equations, see \cite{Krylov book, Ames}. In particular, the lower order
coefficient $|x|^2$, though smooth, is not bounded. This coefficient will drive us to an essential use of Hermite functions
and the Mehler kernel of the Hermite semigroup, see \cite{Thangavelu}. We stress out the fact that we
are able to show the a.e.~convergence of the limits in \eqref{equation x} and \eqref{equation t}
for $p=1$, a result not contained in the previous literature
\cite{Fabes, Fabes-Sadosky, Jones, Krylov book, Krylov0, Ames}.

Our last main result regards parabolic Sobolev estimates with weights for the Cauchy problem
\begin{equation}\label{CauchyP}
\begin{cases}
    \partial_t v=\Delta v-|x|^2v+f,&\hbox{for}~t >0,~x\in\R^n, \\
    v(0,x) = g(x),&\hbox{for}~x \in\R^n.
\end{cases}
\end{equation}
Again our results are presented for \eqref{CauchyP}, though the same statement holds for solutions
to the Cauchy problem for the heat equation, see Subsections \ref{subsection:1} and \ref{subsection:2}.

\begin{thm}[Cauchy problem in the upper half space]\label{fundamental3}
Consider the Cauchy problem \eqref{CauchyP}.
\begin{itemize}
\item[({\bf A})]{\bf Classical solvability.}
Let $g=g(x)$ (resp. $f=f(t,x)$) be a bounded function
with compact support in $\mathbb{R}^n$ (resp., in $\mathbb{R}^{n+1}_+$).
For every $(t,x)\in\mathbb{R}_+^{n+1} $ the integrals  
\begin{equation}\label{duhamel}v(t,x)=\int_0^t\int_{\R^n}\mathcal{W}_\tau(x,y)f(t-\tau,y)\,dy\,d\tau+
\int_{\R^n}\mathcal{W}_t(x,y)g(y)\,dy,
\end{equation}
are well defined.
Moreover, if $f$ and $g$ are $C^2$ functions
then $v$ is a classical solution to \eqref{CauchyP}. In this case the following pointwise limits hold:
 \begin{equation}\label{equation1 x}
\partial_{ij}v(t,x)=\lim_{\varepsilon\to0}\int_0^{t-\varepsilon}\int_{\R^n}\partial_{y_iy_j}\mathcal{W}_\tau(x,y)
f(t-\tau,y)\,dy\,d\tau+\int_{\R^n} \partial_{y_iy_j}\mathcal{W}_t(x,y) g(y) \,dy,
\end{equation}
and
\begin{equation}\label{equation1 t}
\partial_{t}v(t,x) =\lim_{\varepsilon \rightarrow 0}\int_0^{t-\varepsilon} \int_{\mathbb{R}^n}
\partial_{\tau}\mathcal{W}_\tau(x,y) f(t-\tau,y) \,dy\, d\tau+\int_{\mathbb{R}^n}\partial_t\mathcal{W}_t(x,y) g(y)\,dy+f(t,x).
\end{equation}
 \item[({\bf B})] {\bf Weighted parabolic Sobolev estimates.}   
Let  $g=0$ and $f\in L^p(\mathbb{R}_+^{n+1})$, for $1\le p< \infty$, where $w$
belongs to the parabolic Muckenhoupt class $A_p^*(\R^{n+1})$ mentioned above. 
Then the limits \eqref{equation1 x} and \eqref{equation1 t} exist for a.e.~$(t,x) \in \mathbb{R}_+^{n+1}$.
Moreover, the following weighted Sobolev a priori estimates hold: when $1<p<\infty$,
$$\|\partial_{ij}v\|_{L^p(\R^{n+1}_+,w)}+\| \partial_tv\|_{L^p(\R^{n+1}_+,w)}\le
C_{n,p,w}\| f \|_{L^p(\R^{n+1}_+,w)},$$
and, when $p=1$, for any $\lambda>0$,
$$w\big(\{(t,x)\in\R^{n+1}_+:|\partial_{ij}v|+|\partial_tv|>\lambda\}\big)\leq
\frac{C_{n,w}}{\lambda}\| f \|_{L^1(\R^{n+1}_+,w)}.$$
\end{itemize} 
\end{thm}

We notice that Duhamel's principle shows that \eqref{duhamel} is the correct
candidate for solution to \eqref{CauchyP}.
Observe that we get the first integral in \eqref{duhamel} just by restricting the formula of the
solution in the whole space given previously in Theorem \ref{fundamental1}.
In our version of the result above for the heat equation when $g=0$
(Subsection \ref{subsection:2}) we improve Jones' results, compare with \cite{Jones}.

Our key idea is to develop the \textit{language of semigroups} for parabolic equations.
This method allows us to avoid the use of the Fourier transform, which is necessary if we want
to deal with non translation invariant equations and oscillatory integrals.
We should mention that this point of view has been successfully established in recent years to apply to elliptic PDEs, see
\cite{Caffarelli-Stinga, Stempak-Torrea, Stinga-Torrea}, while
a few attempts have also been done for hyperbolic equations, see
\cite{Gale-Miana-Stinga, Kemppainen-Sjogren-Torrea}, as well as for fractional time equations \cite{Bernardis}.
An obvious difference with respect to the elliptic case that produces several technical difficulties 
is the geometry of the underlying space, which is driven by the parabolic distance.
This application of the semigroup language to the heat equation
will give, by following a natural and unified path, all the results in \cite{Fabes, Fabes-Sadosky, Jones}
that we mentioned before. As already shown, new results are obtained
like the a.e.~convergence of the limit \eqref{equation1 x} for functions $f$ in $L^p(\R^{n+1}_+,w)$,
$1\leq p< \infty$, $w\in A_p^*(\R^{n+1})$.

These ideas will be presented with several details in Section \ref{pre}. 
Although part of our results may be known, the reader will see that it is quite convenient to highlight
the structure of the computations in the classical scenario, in particular, by keeping track of
explicit constants. Indeed, the structure persists for the harmonic oscillator evolution equation
and will make the proofs more readable, though quite delicate computations will be needed.
This is due in part to the fact that  the kernel of the heat semigroup is described by  Melher's formula, see \eqref{Hheat},
but also, no less important,  to the fact that we took an accurate account of the constants.
Moreover, those interested just in the heat
equation and the main ideas can just skip Section \ref{pf1}.

We wish to point out another technical point which happens to be crucial.
Along this paper we shall use the vector-valued Calder\'on--Zygmund
theory in spaces of homogeneous type.
We remind the reader that this machinery requires two ingredients:
a kernel satisfying the so-called standard estimates and the
boundedness of the given operator in an $L^{p_0}$ space, for some $1\leq p_0\leq\infty$.
In the cases of Theorems \ref{mixriesz} and \ref{fundamental1},
which correspond to parabolic Riesz transforms, the natural exponent is $p_0=2$.
Theorem \ref{L2} shows the $L^2$ boundedness of the parabolic Riesz transforms
associated to the harmonic oscillator evolution equation,
which is a result of independent interest.
This is consistent with the usual theory of Riesz transforms for the Laplacian, where the Fourier
transform readily shows the $L^2$ continuity.
But in the case of the Poisson operators of Theorem \ref{Poisson}, the natural
initial space is $p_0=\infty$, see the proof of Theorem \ref{Poisson1}.

The path followed to reach the results could be applied to other operators.
It will be clear that our work can be regarded as a unified path to study equations like 
$\partial_tu=Lu+f$, for $f\in L^p(\mathbb{R}\times\Omega)$,
where $L$ is a positive linear differential operator acting on $x\in\Omega$ subject to appropriate boundary conditions. 
The cases in which $L$ is either the Laguerre operator, the Bessel
operator (radial Laplacian), or the Laplace--Beltrami operator
on a Riemannian manifold, as well as other singular integrals like square and area functions,
will be considered in future works.

The organization of the paper is the following.
In Section \ref{pre} we present the crucial formula \eqref{remark 1}
which allows us to define the inverses of our parabolic operators
with the semigroup language approach. From this point on we 
obtain the weighted parabolic Sobolev estimates for the heat equation,
both for the equation posed in the whole space and for the Cauchy problem,
see Theorems \ref{fundamental 1} and \ref{fundamental 3}.
Section \ref{pf1} contains the proof of Theorems \ref{fundamental1} and \ref{fundamental3}.
Finally, Section \ref{mix} is devoted to show Theorems \ref{mixriesz} and \ref{Poisson}.

\section{Weighted mixed-norm Sobolev estimates for the heat equation}\label{pre}

In this section we prove our results for the case of the heat equation.
We will make use of the theory of vector-valued parabolic Calder\'on--Zygmund
singular integrals, which we first describe.

\subsection{The parabolic Calder\'on--Zygmund theory}

Let $\X$ be a set. A function $\rho:\X\times\X\to[0,\infty)$
is called a quasidistance in $\X$ if for any $\x,\y,\z\in\X$ we have: (1) $\rho(\x,\y)=\rho(\y,\x)$, 
(2) $\rho(\x,\y)=0$ if and only if $\x=\y$, and (3) $\rho(\x,\z)\le\kappa(\rho(\x,\y)+\rho(\y,\z))$ for some constant $\kappa\geq1$.
We assume that $\X$ has the topology induced by 
the open balls $B(\x,r)$ with center at $\x\in\X$ and radius $r>0$ defined as $B(\x,r):=\{\y\in\X:\rho(\x,\y)<r\}$.
Let $\mu$ be a positive Borel measure on $(\X,\rho)$
such that, for some universal constant $C_d>0$, we have $\mu(B(\x,2r)) \le C_d\mu(B(\x,r))$ (the so-called doubling property),
for every $\x\in\X$ and $r >0$.
The space $(\X,\rho,\mu)$ is called a \textit{space of homogeneous type}.

Let $w:\X\to\R$ be a weight, namely, a measurable function such that $w(\x)>0$
for $\mu$-a.e. $\x\in\X$.
Given a Banach space $E$, we denote by $L^p_E(\X,w)=L^p(\X,w;E)$, $1\leq p\leq\infty$, the space
of strongly measurable $E$-valued
functions $f$ defined on $\X$ such that $\|f\|_E$ belongs to $L^p(\X,w(\x)d\mu)$.
When $w=1$ we just write $L^p_E(\X)=L^p(\X,E)$.

\begin{defn}[Vector-valued Calder\'on--Zygmund operator on $(\X,\rho,\mu)$]\label{CZ}
Let $E,F$ be Banach spaces. We say that a linear  operator $T$ on 
a space of homogeneous type $(\X,\rho,\mu)$ is a Calder\'on--Zygmund operator if it satisfies the following conditions.
\begin{itemize}
\item [(I)] There exists $1\le p_0 \le\infty$ such that $T$ is bounded from
$L_E^{p_0}(\X)$ into $L_F^{p_0}(\X)$.
\item[(II)] For  bounded $E$-valued functions $f$ with compact support, $Tf$ can be represented as 
\begin{equation}\label{CZ2}
Tf(\x) = \int_X K(\x,\y)f(\y)\,d\mu, \quad\hbox{for}~\x\notin{\rm supp}(f),
\end{equation}
where  $K(\x,\y) \in \mathcal{L} (E,F)$, the space of bounded linear operators from $E$ to $F$ and, moreover, 
\begin{itemize}
\item[(II.1)]  $\displaystyle \|K(\x,\y) \|_{\mathcal{L} (E,F)} \le  \frac{C}{\mu(B(\x,\rho(\x,\y))}$, for every $\x\neq\y$;
\item[(II.2)] $\displaystyle   \|K(\x,\y) - K(\x,\y_0) \|_{\mathcal{L}(E,F)}+\|K(\y,\x) - K(\y_0,\x)\|_{\mathcal{L}(E,F)}
\le C\frac{\rho(\y,\y_0)}{\rho(\x,\y_0)\mu(B(\y_0,\rho(\x,\y_0))},$
whenever $\rho(\x,\y_0)>2\rho(\y,\y_0)$;
\end{itemize}
for some constant $C>0$.
\end{itemize}
\end{defn}
In this paper we shall be mainly working with the space of homogeneous type $(\X,\rho,\mu)=(\mathbb{R}^{n+1},d,dtdx)$,
where $d$ is the parabolic distance defined by
\begin{equation}\label{eq:parabolic distance}
d\big((t,x),(s,y)\big)=\max(|t-s|^{1/2},|x-y|),\quad\hbox{for}~(t,x),(s,y)\in\R^{n+1},
\end{equation}
and $dtdx$ is the Lebesgue measure
on $\R^{n+1}$. Observe that in this case, $B((t,x),r)=\{(s,y)\in\R^{n+1}:\max(|t-s|^{1/2},|x-y|)<r\}$ and
$|B((t,x),r)|=|B((0,0),r)|\sim r^{n+2}$,
so $dtdx$ is a doubling measure on parabolic balls as required.
On the other hand, it is clear that $\R^n$ with the usual Euclidean distance and the Lebesgue measure is a space 
of homogeneous type.

A weight $w$ on $(\X,\rho,\mu)$ is said to be a Muckenhoupt $A_p(\X)$ weight, $1<p<\infty$,
if there exists a constant $C_p>0$ such that
$$\bigg(\frac{1}{\mu(B)}\int_Bw(\x)\,d\mu\bigg)\bigg(\frac{1}{\mu(B)}\int_Bw(\x)^{1/(1-p)}\,d\mu\bigg)^{p-1}\leq C_p,$$
for every ball $B\subset\X$. The weight belongs to $A_1(\X)$ if there is a constant $C_1>0$ such that
$$\frac{1}{\mu(B)}\int_{B}w(\y)\,d\mu(\y)\leq C_1w(\x),$$
for every ball $B\subset\X$ that contains $\x$, for a.e. $\x\in\X$.

The Calder\'on--Zygmund Theorem says that if $T$ is a Calder\'on--Zygmund
operator on a space of homogeneous type $(\X,\rho,\mu)$ as above
then $T$ is bounded from $L^p_E(\X,w)$ into $L^p_F(\X,w)$, for any $1<p<\infty$ and $w\in A_p(\X)$,
and it is also bounded from $L^1_E(\X,w)$ into
weak-$L^1_F(\X,w)$, for any $w\in A_1(\X)$. Moreover, the maximal operator of the truncations
\begin{equation}\label{truncation}
T^\ast f(\x) =\sup_{\varepsilon>0}\bigg\| \int_{\rho(\x,\y)>\varepsilon} K(\x,\y)f(\y)\,d\mu\bigg\|_F,
\end{equation}
is bounded from $L^p_E(\X,w)$ into $L^p(\X,w)$, $w\in A_p(\X)$, for $1<p<\infty$, and from $L^1_E(\X,w)$
into weak-$L^1(\X,w)$, $w\in A_1(\X)$.

For full details about the theory presented above see \cite{MST, RRT, Francisco, RT}.

 Notice next that for the case of the parabolic distance \eqref{eq:parabolic distance} the right hand sides in conditions
(II.1) and (II.2) above read, for $\x=(t,x)$, $\y=(s,y)$ and $\y_0=(s_0,y_0)$,
\begin{equation}\label{first}
\frac{C}{\mu(B(\x,\rho(\x,\y))}=\frac{C}{\max(|t-s|^{1/2},|x-y|)^{n+2}}\sim\frac{C}{(|t-s|^{1/2}+|x-y|)^{n+2}},
\end{equation}
and
\begin{equation}\label{second}
\frac{\rho(\y,\y_0)}{\rho(\x,\y_0)\mu(B(\y_0,\rho(\x,\y_0))}=
\frac{\max(|s-s_0|^{1/2},|y-y_0|)}{\max(|t-s_0|^{1/2},|x-y_0|)^{n+3}}\sim
\frac{|s-s_0|^{1/2}+|y-y_0|}{(|t-s_0|^{1/2}+|x-y_0|)^{n+3}},
\end{equation}
respectively. Finally, the set of points $\y\in\X$ such that $\rho(\x,\y)>\varepsilon$ appearing in \eqref{truncation} is
\begin{equation}\label{conjuntitos}
\Omega_\varepsilon(t,x):=\{(s,y)\in\R^{n+1}:\max(|t-s|^{1/2},|x-y|)>\varepsilon\}.
\end{equation}

\subsection{The semigroup language and the heat equation}

As the operators $\partial_t$ and $\Delta$ commute, the semigroup $\{e^{-\tau(\partial_t-\Delta)}\}_{\tau\geq0}$
is given by the composition
$e^{-\tau(\partial_t-\Delta)}=e^{-\tau \partial_t}\circ e^{\tau\Delta}$.
In particular, for smooth functions $\varphi(t,x)$ with rapid decay at infinity we have
$$e^{-\tau(\partial_t-\Delta)}\varphi(t,x)=e^{\tau\Delta}\varphi(t-\tau,x)=\int_{\R^n}W(\tau,y)\varphi(t-\tau,x-y)\,dy,$$
where $W(\tau,y)$ denotes the usual Gauss--Weierstrass kernel
\begin{equation}\label{GW}
W(\tau,y):= \frac1{(4\pi\tau)^{n/2}}e^{-|y|^2/(4\tau)}.
\end{equation}
Recall that $\partial_\tau W-\Delta_yW=0$, for $\tau>0$ and $y\in\R^n$.
Notice that 
\begin{equation}\label{Fheat}
\widehat{e^{-\tau(\partial_t-\Delta)}\varphi}(\rho,\xi)=e^{-\tau(i\rho+|\xi|^2)}\widehat{\varphi}(\rho,\xi),
\end{equation}
for $\rho\in\R$, $\xi\in\R^n$.
On the other hand it is easy to check that 
\begin{equation}\label{remark 1}
(i\rho +\lambda)^{-1}=\int_0^\infty e^{-\tau(i\rho+ \lambda)}\,d\tau\quad\hbox{for}~\lambda >0,
\end{equation}
where the integral is absolutely convergent, see for example \cite{Stinga-Torrea-s}. 
By using formulas \eqref{remark 1} and \eqref{Fheat} we define, for any $t\in\R$, $x\in\R^n$,
$$(\partial_t-\Delta)^{-1}f(t,x)= \int_0^\infty e^{-\tau(\partial_t-\Delta)}f(t,x)\,d\tau=
\int_0^\infty\int_{\R^n}\frac{e^{-|y|^2/(4\tau)}}{(4\pi\tau)^{n/2}} f(t-\tau,x-y)\,dy\,d\tau.$$

\begin{rem}
In fact, these two ideas can be used to find formulas in order to solve parabolic equations of the form
$\partial_tu+Lu=f$, for $f\in L^2(\mathbb{R}\times\Omega)$,
where $L$ is a nonnegative densely defined self-adjoint
linear operator in some $L^2(\Omega,d\eta)$. This observation will be crucial in Section \ref{pf1}.
\end{rem}

\subsection{Heat equation: classical solvability and weighted Sobolev estimates in the whole space}\label{subsection:1}

In this subsection we solve the heat equation in the whole space and we prove the 
weighted Sobolev estimates. This is the heat equation counterpart of Theorem \ref{fundamental1}.

\begin{thm}\label{fundamental 1}
Let $W(\tau,y)$, $\tau>0$, $y\in\R^n$, be the Gauss--Weierstrass kernel \eqref{GW}.
\begin{itemize}
\item[({\bf A})]{\bf Classical solvability.}
Let $f=f(t,x)$ be a bounded function with compact support on $\R^{n+1}$.
Then for every $(t,x) \in \mathbb{R}^{n+1}$ the following integral is well defined 
$$u(t,x)=\int_0^\infty \int_{\mathbb{R}^n}W(\tau,y)f(t-\tau,x-y)\,dy\,d\tau.$$
If $f$ is also a $C^2$ function then the function $u$ as defined above is a classical solution to the heat equation
$\partial_tu=\Delta u+f$, in $\R^{n+1}$.
Moreover the following pointwise limit formulas hold:
\begin{equation}\label{equation2 x}
\partial_{ij}u(t,x)=\lim_{\varepsilon\rightarrow0}\iint_{\Omega_\varepsilon}
\partial_{y_iy_j}W(\tau,y)f(t-\tau,x-y)\,dy\,d\tau-A_nf(t,x)\delta_{ij},
\end{equation}
and
\begin{equation}\label{equation2 t}
\partial_t u(t,x)=\lim_{\varepsilon\rightarrow0}\iint_{\Omega_\varepsilon}
\partial_\tau W(\tau,y)f(t-\tau,x-y)\,dy\,d\tau+B_n f(t,x),
\end{equation}
where
\begin{equation}\label{eq:AyB}
A_n = \frac1{n \Gamma(\frac{n}{2})}\int_{1/4}^\infty w^\frac{n}{2} e^{-w} \,\frac{dw}{w},
\qquad B_n = \frac1{ \Gamma(\frac{n}{2})}\int_0^{1/4}w^\frac{n}{2} e^{-w} \,\frac{dw}{w},
\end{equation}
and $\Omega_\varepsilon= \{(\tau,y):\max(|\tau|^{1/2},|y|)>\varepsilon\}$.
\item[({\bf B})]{\bf Weighted parabolic Sobolev estimates.}
In the case when $f\in L^p(\mathbb{R}^{n+1}, w)$, $1\le p<\infty$, $w\in A_p^*(\R^{n+1})$, the limits above
exist for a.e.~$(t,x)\in\R^{n+1}$ and the following a priori estimates hold: for $1<p<\infty$,
$$\|\partial_{ij}u\|_{L^p(\R^{n+1},w)}+\|\partial_t u\|_{L^p(\R^{n+1},w)}\leq C_{n,p,w}\|f\|_{L^p(\R^{n+1},w)},$$
and, in the case $p=1$, for any $\lambda>0$,
$$w\big(\{(t,x)\in\R^{n+1}:|\partial_{ij}u|+|\partial_tu|>\lambda\}\big)
\le\frac{C_{n,w}}{\lambda}\|f\|_{L^1(\R^{n+1},w)}.$$
\end{itemize}
\end{thm}

\begin{proof}[Proof of Theorem \ref{fundamental 1} Part (A)]
For any $(t,x)\in\R^{n+1}$ there exists a constant $c_t$ depending on the support of $f$ such that 
$$\bigg|\int_0^\infty \int_{\R^n}W(\tau,x-y)f(t-\tau,y)\,dy\,d\tau\bigg|
\le C\|f\|_{L^\infty(\R^{n+1})}\int_0^{c_t}\int_{\R^n}W(\tau,y)\,dy\,d\tau=C_{n,t},$$
where we used that the integral in $y$ of the kernel $W(\tau,y)$ is identically $1$ for any $\tau$.
The argument above also shows that we can interchange the integral
and the second derivatives $\partial_{ij}$ when $f$ is a $C^2$ function with compact support.
 
Next we show that $u(t,x)$ satisfies the equation.  We first compute $\partial_{ii}u$. Observe that 
for any $i=1,\ldots,n$,
\begin{align*}
\partial_{ii}u(t,x) &= \int_0^\infty \int_{\R^n}W(\tau,y)\partial_{x_ix_i}f(t-\tau,x-y)\,dy\ d\tau \\
&= \lim_{\varepsilon \rightarrow 0 }\iint_{\Omega_\varepsilon}W(\tau,y)\partial_{y_iy_i}f(t-\tau,x-y)\,dy\,d\tau.
\end{align*}
where $ \Omega_\varepsilon=\{(\tau,y):\max(\tau^{1/2},|y|)>\varepsilon\}$. Integration by parts gives
\begin{align*}
\iint_{\Omega_\varepsilon}W(\tau,y)\partial_{y_iy_i}f(t-\tau,x-y)\,dy\,d\tau&=
  - \iint_{\Omega_\varepsilon}\partial_{y_i}W(\tau,y)\partial_{y_i}f(t-\tau,x-y)\,dy\,d\tau \\
  &\quad+\iint_{\partial{\Omega_\varepsilon}}W(\tau,y)\partial_{y_i}f(t-\tau,x-y)\nu_i\,d\sigma(y,\tau),
\end{align*}
where $\nu_i$ is the $i$th-component of the exterior unit normal vector to $\partial{\Omega_\varepsilon}$.
Let us write
\begin{equation}\label{eq:omega1}
\partial{\Omega_\varepsilon} = \partial{\Omega_\varepsilon^1}\cup
\partial{\Omega_\varepsilon^2}\cup\partial{\Omega_\varepsilon^3},
\end{equation}
where (recall that $\tau>0$)
\begin{equation}\label{eq:omega2}
\begin{aligned}
\partial{\Omega_\varepsilon^1}&=\{(\tau,y): \, |y| <  \varepsilon,\,\tau^\frac{1}{2} = \varepsilon\},\\
\partial{\Omega_\varepsilon^2}&= \{(\tau,y): \, |y| = \varepsilon, 0< \tau^\frac{1}{2} <\varepsilon \},\\
\partial{\Omega_\varepsilon^3}&= \{(\tau,y): \,|y| < \varepsilon,\, \tau^\frac{1}{2}= 0 \}.
\end{aligned}
\end{equation}
Observe that the exterior unit normal vector on
$\partial{\Omega_\varepsilon^1}$ is $(-1,0,\ldots,0)\in\R^{n+1}$. Then
$$\iint_{\partial{\Omega_\varepsilon^1}}W(\tau,y)\partial_{y_i}f(t-\tau,x-y)\nu_i\,d\sigma(y,\tau)= 0,$$
for any $i=1,\dots,n$, and the same is true for the boundary integral over $\Omega_\varepsilon^3$. 
On the other hand, the unit normal  of $\partial{\Omega_\varepsilon^2}$ is $\frac{1}{\varepsilon}(0,-y)$. Hence
\begin{align*}
\iint_{\partial{\Omega_\varepsilon^2}}W(\tau,y)|\partial_{y_i}f(t-\tau,x-y)|\,d\sigma(y,\tau)
&\leq C\int_0^{\varepsilon^2}\int_{|y|=\varepsilon} \frac{e^{-|y|^2/(4\tau)}}{\tau^{n/2}}\,d\sigma(y)\,d\tau\\
&= C\int_0^{\varepsilon^2}\frac{e^{-\varepsilon^2/(4\tau)}}{\tau^{n/2}}\varepsilon^{n-1}\,d\tau  
=C\varepsilon\rightarrow0,\quad\hbox{as}~\varepsilon\rightarrow0.
\end{align*}
Again, integration by parts together with a parallel discussion of the boundary integrals gives
\begin{align*}
-\iint_{\Omega_\varepsilon}\partial_{y_i}W(\tau,y)\partial_{y_i}f(t-\tau,x-y)\,dy\,d\tau
&=\iint_{\Omega_\varepsilon}\partial_{y_iy_i}W(\tau,y)f(t-\tau,x-y)\,dy\,d\tau\\
&\quad-\iint_{\partial{\Omega^2_\varepsilon}}\partial_{y_i}W(\tau,y)f(t-\tau,x-y)\nu_i\,d\sigma(y,\tau) \\
&=:I_1-I_2.
\end{align*}
The integral $I_1$ corresponds to the first term in \eqref{equation2 x} when $i=j$. Let us rewrite $I_2$ as
\begin{align*}
I_2 &=- \iint_{\partial{\Omega_\varepsilon^2}}\partial_{y_i}W(\tau,y)\big(f(t-\tau,x-y)-f(t,x)\big)\frac{y_i}{|y|}\,d\sigma(y,\tau)
-f(t,x)\iint_{\partial{\Omega_\varepsilon^2}}\partial_{y_i}W(\tau,y)\frac{y_i}{|y|}\,d\sigma(y,\tau) \\
&=: I_{21}+I_{22}.
\end{align*}
Since
$$\partial_{y_i}W(\tau,y)=-\frac{e^{-|y|^2/(4\tau)}}{(4 \pi \tau)^{n/2}}\cdot\frac{y_i}{2\tau},$$
by the Mean Value Theorem we get
\begin{align*}
|I_{21}| &\leq C\int_0^{\varepsilon^2} \int_{|y|=\varepsilon} \frac{|y|}{\tau^{n/2+1}}e^{-|y|^2/(4\tau)}(\tau+|y|)\,d\sigma(y)\,d\tau\\
&\leq C\int_0^{\varepsilon^2}\frac{ \varepsilon^{n+1}e^{-\varepsilon^2/(4\tau)}}{\tau^{n/2+1}}\,d\tau
=C\varepsilon\rightarrow0,\quad\hbox{as}~\varepsilon\rightarrow0,
\end{align*}
where we have assumed $\varepsilon<1$. For $I_{22}$ we notice that its value is independent of $i$, so by 
taking the sum over $i=1,\ldots,n$ we obtain
\begin{align*}
I_{22} &= \frac{1}{n} f(t,x)\int_0^{\varepsilon^2} \int_{|y|=\varepsilon}
\frac{ e^{-|y|^2/(4\tau)}}{(4 \pi \tau)^{n/2}}\cdot\frac{|y|}{2\tau}\,d\sigma(y)\,d\tau
=\frac{|S^{n-1}|f(t,x)}{2n(4 \pi)^{n/2}}\int_1^\infty r^\frac{n}{2}e^{-\frac{r}{4}}\,\frac{dr}{r}.
\end{align*}
Pasting together our last computations we arrive to 
\begin{equation}\label{identity x}
\partial_{ii}u(x,t)= \lim_{\varepsilon\rightarrow 0}\int\int_{\Omega_\varepsilon}
\partial_{y_iy_i}W(\tau,y)f(t-\tau,x-y)\,dy\,d\tau-A_nf(t,x),
\end{equation}
with $A_n$ as in \eqref{eq:AyB}.

Observe that in the parallel
computation for $\partial_{ij}u$ with $i\neq j$ the integral in the term $I_{22}$ will be equal to 
$$\iint_{\partial \Omega_\varepsilon^2}\partial_{y_j}W(\tau,y)\frac{y_i}{|y|}\,d\sigma(\tau,y)
= \int_0^{\varepsilon^2}\int_{|y|=\varepsilon}\frac{e^{-|y|^2/(4\tau)}}{(4\pi \tau)^{n/2}}\frac{y_iy_j}{2\tau|y|}\,d\sigma(y)\,d\tau=0.$$
Then \eqref{equation2 x} is true.

Next we compute $\partial_tu(t,x)$. In a similar fashion as before,
\begin{multline*}
\partial_tu(t,x)=-\lim_{\varepsilon\rightarrow 0}\iint_{\Omega_\varepsilon}W(\tau,y)\partial_\tau f(t-\tau,x-y)\,dy\,d\tau\\
=\lim_{\varepsilon\rightarrow 0}\iint_{\Omega_\varepsilon}\partial_\tau W(\tau,y)f(t-\tau,x-y)\,dy\,d\tau
-\lim_{\varepsilon\rightarrow 0}\iint_{\partial{\Omega_\varepsilon}}W(\tau,y)f(t-\tau,x-y)\nu_{n+1}\,d\sigma(y,\tau). 
\end{multline*}
Again, we decompose $\partial\Omega_\varepsilon$ as in \eqref{eq:omega1}--\eqref{eq:omega2}.
Clearly, $\nu_{n+1}=0$ on $\partial\Omega^2_\varepsilon$. On the other hand,
$$\iint_{\partial{\Omega^3_\varepsilon}}W(\tau,y)f(t-\tau,x-y)\,d\sigma(y,\tau)
=\int_{|y|=\varepsilon}W(0,y)f(t,x-y)\,d\sigma(y)=0.$$
Parallel to the spatial derivatives case we write
\begin{align*}
\iint_{\partial{\Omega^1_\varepsilon}}W(\tau,y)f(t-\tau,x-y)\nu_\tau\,d\sigma(y,\tau) 
&= \iint_{\partial{\Omega^1_\varepsilon}}W(\tau,y)\big(f(t-\tau,x-y)-f(t,x)\big)\nu_\tau\,d\sigma(y,\tau) \\
&\quad+f(t,x)\iint_{\partial{\Omega^1_\varepsilon}}W(\tau,y)\nu_\tau\,d\sigma(y,\tau)=:J_1+J_2
\end{align*}
We apply the Mean Value Theorem in $J_1$ to get
\begin{align*}
|J_1| &\leq C \int_{\tau=\varepsilon^2}\int_{|y|< \varepsilon} \frac{e^{-|y|^2/(4\tau)}}{(4 \pi \tau)^{n/2}}(\tau+|y|)\,d\sigma(y,\tau)\\
&\leq C\varepsilon^{(1-n)}\int_0^\varepsilon r^{n-1}e^{-r^2/(4\varepsilon^2)}\,dr
= C\varepsilon\int_0^1w^{n-1}e^{-w^2/4}\,dw\rightarrow 0,\quad\hbox{as}~\varepsilon\rightarrow0,
\end{align*}
where we have assumed that $\varepsilon<1$. Finally, for $J_2$, we have
$$J_2 = -\frac{|S^{n-1}|f(t,x)}{(4 \pi\varepsilon^2)^{n/2}}\int_0^\varepsilon r^{n-1}e^{-r^2/(4\varepsilon^2)}\,dr
 = -\frac{|S^{n-1}|f(t,x)}{(4\pi)^{n/2}}\int_0^1 w^n e^{-\frac{w^2}{4}}\,\frac{dw}{w}.$$
In other words, we have \eqref{equation2 t}
with $B_n$ as in \eqref{eq:AyB}. From \eqref{identity x} and \eqref{equation2 t} we get $\partial_tu=\Delta u+f$.
\end{proof}

\begin{proof}[Proof of Theorem \ref{fundamental 1} Part (B)]
The identities \eqref{equation2 x} and \eqref{equation2 t} establish that the
parabolic Riesz transforms
$$R_{ij}^\Delta:=\partial_{ij}(\partial_t-\Delta)^{-1}\quad\hbox{and}\quad R_t^\Delta:=\partial_t(\partial_t-\Delta)^{-1},$$
for $i,j=1,\ldots,n$, can be seen as operators satisfying \eqref{CZ2} in Definition \ref{CZ}.
On the other hand, for functions $f\in L^2(\mathbb{R}^{n+1})$ we have 
$$\widehat{R_{ij}^\Delta f}(\rho,\xi)=\frac{\xi_i\xi_j}{i\rho+|\xi|^2}\widehat{f}(\rho,\xi)\quad
\hbox{and}\quad\widehat{R_t^\Delta f}(\rho,\xi)= \frac{i\rho}{i\rho+|\xi|^2} \widehat{f}(\rho,\xi).$$
The Fourier multipliers above are bounded functions, hence the parabolic Riesz transforms
are bounded operators in $L^2(\R^{n+1})$. In order to be able to conclude the weighted $L^p$ boundedness
and the weighted weak $(1,1)$ type estimate, we have to verify
that the kernels satisfy the size and smoothness conditions described in Definition \ref{CZ}, see also
\eqref{first} and \eqref{second}.
We show how to do this for $R_{ij}$, the case of $R_t$ follows similar lines.

We first observe that the kernel
$W(\tau,y)$ in \eqref{GW} is defined for $(\tau,y) \in \mathbb{R}^{n+1}_+ \setminus \{(0,0)\}$.
We can extend this kernel to the whole space $\mathbb{R}^{n+1}\setminus \{(0,0)\}$ just by setting
$W(\tau,y)=0$, for $\tau\le0$ and $y\in\R^n\setminus\{0\}$.
Observe that this extended kernel is a smooth function in the $\tau$ and $y$ variables, whenever $(\tau,y)\neq(0,0)$.
Now the kernels of the operators $R_{ij}$ can be computed by taking the corresponding derivatives
of the above extended function $W$.
In order to get the size and the smoothness conditions of the kernels it is enough to get them for $\tau >0$.

The identity in \eqref{equation2 x} shows that the integral kernel of $R_{ij}$ is given by
$$R_{ij}^\Delta(\tau,y)=\partial_{y_iy_j}W(\tau,y)= \bigg(-\frac1{2\tau}\delta_{ij}+\frac{y_iy_j}{4\tau^2}\bigg)
\frac{e^{-|y|^2/(4\tau)}}{(4\pi\tau)^{n/2}},$$
for $\tau>0$ and $y\in\R^n$.  Then we get 
$$|R_{ij}^\Delta(\tau,y)|\leq C \frac{ e^{-|y|^2/(c\tau)}}{\tau^{n/2+1}}, \quad (\tau,y)\neq (0,0).$$
By taking into account the cases $|y|+\tau^{1/2}\le2\tau^{1/2}$
and $|y|+\tau^{1/2}>2\tau^{1/2}$ it is easy to see that
$$|R_{ij}^\Delta(\tau,y)|\leq\frac{C}{(\tau^{1/2}+|y|)^{n+2}}, \quad (\tau,y)\neq (0,0).$$
The other identities follow in a similar way.
We also observe that the sets $\Omega_\varepsilon$ appearing in part (A)
correspond to the truncations $\Omega_\varepsilon(0,0)$ in \eqref{conjuntitos} and \eqref{truncation}.
We leave to the reader to complete the rest of the proof.
\end{proof}

\subsection{Heat equation: classical solvability and weighted Sobolev estimates for the Cauchy problem}\label{subsection:2}

In this subsection we prove the heat equation counterpart of Theorem \ref{fundamental3}.

\begin{thm}\label{fundamental 3}
Consider the Cauchy problem for the heat equation
\begin{equation}\label{CauchyPL}
\begin{cases}
\partial_tv=\Delta v+f,&\hbox{for}~t >0,~x\in\R^n,\\
v(0,x)=g(x),&\hbox{for}~x\in\R^n.
\end{cases}
\end{equation}
\begin{itemize}
\item[({\bf A})]{\bf Classical solvability.}
Let $g=g(x)$ (resp. $f=f(t,x)$) a bounded function
with compact support in $\mathbb{R}^n$ (resp. in $\mathbb{R}^{n+1}_+$).
Then for every $(t,x)\in\mathbb{R}_+^{n+1}$ the integrals  
$$v(t,x)=\int_0^t\int_{\R^n}W(\tau,y)f(t-\tau,x-y)\,dy\,d\tau+\int_{\R^n}W(t,x-y)g(y)\,dy,$$
are well defined. If $f$ is also a $C^2$ function then $v$ is a classical solution to
\eqref{CauchyPL}. Moreover the following pointwise limit formulas hold:
\begin{equation}\label{equation3 x}
\partial_{ij}v(t,x)=\lim_{\varepsilon\rightarrow 0}\int_\varepsilon^t\int_{\mathbb{R}^n}
\partial_{y_iy_j}W(\tau,x-y)f(t-\tau,y)\,dy\,d\tau+\int_{\mathbb{R}^n}\partial_{y_iy_j}W(t,y)g(x-y)\,dy,
\end{equation}
and
\begin{equation}\label{equation3 t}
\partial_{t}v(t,x) =\lim_{\varepsilon \rightarrow 0}\int_\varepsilon^t \int_{\mathbb{R}^n} \partial_{\tau}W(\tau,x-y)f(t-\tau,y) \,dy\, d\tau
+\int_{\mathbb{R}^n}\partial_tW(t,y) g(x-y)\,dy+f(t,x).
\end{equation}
\item[({\bf B})]{\bf Weighted parabolic Sobolev estimates.}
In the case when $g=0$ and $f\in L^p(\mathbb{R}_+^{n+1},w)$,
$w\in A^*_p(\R^{n+1})$,
for some $1\le p< \infty$,
the limits above exist for a.e. $(t,x) \in \mathbb{R}_+^{n+1}$ and the following a priori estimates hold: for $1<p<\infty$,
$$\|\partial_{ij}v\|_{L^p(\R^{n+1}_+,w)}+\| \partial_tv\|_{L^p(\R^{n+1}_+,w)}
\le C_{n,p,w}\|f\|_{L^p(\R^{n+1}_+,w)},$$
and, in the case $p=1,$ for any $\lambda>0$,
$$w\big(\{(t,x)\in\R^{n+1}_+:|\partial_{ij}v|+|\partial_{t}v|>\lambda\}\big)
\le\frac{C_{n,p,w}}{\lambda}\|f \|_{L^1(\R^{n+1}_+,w)}.$$
\end{itemize}
\end{thm}

\begin{proof}[Proof of Theorem \ref{fundamental 3} Part (A)] 
Observe that by linearity it is enough to solve the problems 
$$\begin{cases} 
\partial_tv^1-\Delta v^1=0,&\hbox{for}~t >0,~x\in\R^n,\\
v^1(0,x)=g(x),&\hbox{for}~x\in\mathbb{R}^n,
\end{cases}$$
and
$$\begin{cases} 
\partial_tv^2-\Delta v^2=f,&\hbox{for}~t >0,~x\in\R^n,\\
v^2(0,x)=0,&\hbox{for}~x\in\mathbb{R}^n,
\end{cases}$$
so $v=v_1+v_2$. We deal with $v_1$ and $v_2$ separately.

On one hand, the solution $v_1$ is given by the Gauss--Weierstrass semigroup $v^1(t,x)=e^{t\Delta}g(x)$.
This produces all the terms and properties in the statement related to the initial datum $g$.

The second problem will be solved by following the steps of the proof of Theorem \ref{fundamental 1}
with the appropriated changes due to the nature of the new ambient space $\mathbb{R}^{n+1}_+$.
We start as in the proof of Theorem \ref{fundamental 1} but replacing the set
$\Omega_\varepsilon$ by the set $\Sigma_\varepsilon = \{(\tau,y): \tau > \varepsilon\}$.
It is easy to check that integration by parts produces the first term in formula \eqref{equation3 x}.
For the derivative with respect to $t$, observe that $\partial \Sigma_\varepsilon= \{(\tau,y): \tau=\varepsilon\}$.
Then  by using parametric derivation and integration by parts we get  
\begin{align*}
\partial_t v^2(t,x) &= \lim_{\varepsilon \rightarrow 0}\int_\varepsilon^t\int_{\mathbb{R}^n}\partial_\tau{W}(\tau,y)f(t-\tau,x-y)\,d\tau\,dy\\
&\quad+\lim_{\varepsilon \rightarrow 0}\int_{\mathbb{R}^n}{W}(\varepsilon,y)f(t-\varepsilon,x-y)\,dy
- \int_{\mathbb{R}^n}{W}(t,y)f(0,x-y)\,dy \\
&= \lim_{\varepsilon \rightarrow 0}\int_\varepsilon^t \int_{\mathbb{R}^n}\partial_\tau{W}(\tau,y) f(t-\tau,x-y)\,d\tau\,dy
+\lim_{\varepsilon \rightarrow 0}  \int_{\mathbb{R}^n}{W}(\varepsilon,y)f(t,x-y)\,dy\\
&\quad+\lim_{\varepsilon \rightarrow 0}\int_{\mathbb{R}^n}{W}(\varepsilon,y)\big(f(t-\varepsilon,x-y)-f(t,x-y)\big)\,dy\\
&=  \lim_{\varepsilon \rightarrow 0} \int_\varepsilon^t \int_{\mathbb{R}^n} \partial_\tau{W}(\tau,y) f(t-\tau,x-y)\,d\tau\,dy+f(t,x).
\end{align*}
\end{proof}

\begin{proof}[Proof of Theorem \ref{fundamental 3} Part (B)] 
Notice that 
the set $\Sigma_\varepsilon$ does not correspond to the standard truncations for Calder\'on--Zygmund operators,
see \eqref{truncation} and \eqref{conjuntitos}. 
Therefore we can not apply the Calder\'on--Zygmund machinery for the whole space.
In order to prove the results we will do a comparison argument with the global case.
Given $f\in L^p(\mathbb{R}^{n+1}_+),$ consider the difference
$$\bigg|\iint_{\Omega_\varepsilon }\partial_{y_iy_j}{W}(\tau,y) f(t-\tau,x-y)\chi_{0<\tau<t} \,dy\, d\tau- 
\int_{\varepsilon^2}^t \int_{\mathbb{R}^n}\partial_{y_iy_j}{W}(\tau,y) f(t-\tau,x-y) \,dy\, d\tau\bigg|.$$
We have 
\begin{align*}
\big|(\chi_{\Omega_\varepsilon}-\chi_{\Sigma_{\varepsilon^2}})\chi_{\tau<t}\partial_{y_iy_j}{W}(\tau,x-y)\big|
&\le \chi_{|y|>\varepsilon}\chi_{\tau < \varepsilon^2}\frac1{(|y|+\tau^{1/2})^{n+2}}\\
&=\chi_{\frac{|y|}{\varepsilon}>1}\chi_{\frac{\tau}{\varepsilon^2}<1}
\frac1{\varepsilon^{n+2}\big(\frac{|y|}{\varepsilon}+\frac{\tau^{1/2}}{\varepsilon}\big)^{n+2}}  
=\frac1{\varepsilon^{n+2}}\Psi\bigg(\frac{y}{\varepsilon},\frac{\tau^{1/2}}{\varepsilon}\bigg),
\end{align*}
with $\Psi(z,s^{1/2})=(|z|+s^{1/2})^{-(n+2)}\chi_{|z| >1}\chi_{s<1}$. It is easy to see that
$\displaystyle\int_{\mathbb{R}^{n+1}} \Psi(z,s^{1/2})\,dz\,ds\leq C$.
Therefore
\begin{multline*}
\sup_\varepsilon\bigg|\iint_{\Omega_\varepsilon }\partial_{y_iy_j}{W}(\tau,y) f(t-\tau,x-y)\chi_{0<\tau<t} \,dy\, d\tau- 
\int_{\varepsilon^2}^t \int_{\mathbb{R}^n}\partial_{y_iy_j}{W}(\tau,y) f(t-\tau,x-y) \,dy\, d\tau\bigg|\\
\le \sup_\varepsilon \int_{\mathbb{R}^{n+1}} \frac1{\varepsilon^{n+2}}\Psi\bigg(\frac{x-y}{\varepsilon},
\frac{\tau^{1/2}}{\varepsilon}\bigg)|f(t-\tau,y)|\,dy\,d\tau,
\end{multline*}
and this operator is bounded from $L^p(\R^{n+1},w)$ into itself
for $w\in A_p^*(\R^{n+1})$ and from $L^1(\R^{n+1},w)$ into weak-$L^1(\R^{n+1},w)$ for weights $w\in A_1^*(\R^{n+1})$.
Now we remind that for good enough functions we have, see \eqref{equation2 x} and \eqref{equation3 x},
$$\lim_{\varepsilon \rightarrow 0}\bigg\{\iint_{\Omega_\varepsilon}-\iint_{\Sigma_{\varepsilon^2}}\bigg\}\,
\partial_{y_iy_j}{W}(\tau,y)f(t-\tau,x-y)\,d\tau\,dy=\delta_{ij}A_nf(t,x).$$
An application of the Banach principle of almost everywhere convergence
gives the proof of statement \eqref{equation3 x} in Theorem \ref{fundamental 3}.
For \eqref{equation3 t} we can proceed similarly, details are left to the interested reader.
\end{proof}

\section{Weighted estimates and a.e. convergence: proof of Theorems \ref{fundamental1} and \ref{fundamental3}}\label{pf1}

In this section we prove the Sobolev estimates for the parabolic harmonic oscillator equation.
The structure of the proofs and several computations follow a parallel path to that of Section \ref{pre}.

\subsection{Basic preliminaries}

Let $h_k(r)$, $r\in\R$, be the collection of Hermite functions on the real line: 
$$h_k(r) = (\pi^{1/2} 2^k k! )^{-1/2} H_k(r)e^{-r^2/2},\quad k =0,1,2, \dots.$$
Here $H_k(r)$ denotes the classical Hermite polynomial of degree $k$.
The multidimensional Hermite functions are
$h_\alpha(x)=h_{\alpha_1}(x_1)\cdots h_{\alpha_n}(x_n)$, for $\alpha=(\alpha_1,\dots,\alpha_n)\in\N^n_0$,
$x\in\R^n$. Let $H=-\Delta+|x|^2$ be the harmonic oscillator operator, which is a positive
and symmetric operator in $L^2(\R^n)$ with domain $C^\infty_c(\R^n)$. It is well known that 
the Hermite functions give the spectral decomposition of $H$ in $L^2(\R^n)$ with
$Hh_\alpha=(2|\alpha|+n)h_\alpha$, where $|\alpha|=\alpha_1+\cdots+\alpha_n$.
The heat semigroup $\{e^{-\tau H}\}_{\tau>0}$ is given by integration against a kernel, see \cite{Thangavelu}.
Indeed, for functions $\varphi \in L^p(\mathbb{R}^n)$,
\begin{equation}\label{Hheat}
\begin{aligned}
e^{-\tau H}\varphi(x)&= \int_{\mathbb{R}^n}\mathcal{W}_\tau(x,y)\varphi(y)\,dy \\
&=\int_{\mathbb{R}^n} \frac{1}{(2\pi \sinh 2\tau)^{n/2}}
\exp\big(-\tfrac{1}{2}|x-y|^2\coth 2\tau - x\cdot y \tanh \tau \big)\varphi(y)\,dy.
\end{aligned}
\end{equation}
As in the previous sections, we define
\begin{equation}\label{FHermite}
(\partial_t +H)^{-1}=\int_0^\infty e^{-\tau \partial_t} \circ e^{-\tau H}\,d\tau.
\end{equation}

\subsection{Proof of Theorem \ref{fundamental1} Part (A)}

Now we shall start with the proof of Theorem \ref{fundamental1}, that 
follows closely the proof of Theorem \ref{fundamental 1}.    
Assume that $f=f(t,x)$ is a bounded function with compact support. As
$\coth2\tau= \frac{1+(\coth \tau)^2}{2\coth \tau }$,
we have, from \eqref{FHermite} and \eqref{Hheat},
\begin{equation}\label{fundamental 2}
\begin{aligned}
(\partial_t&+H)^{-1}f(t,x) =  \int_0^\infty\int_{\mathbb{R}^n}\mathcal{W}_\tau(x,y)f(t-\tau,y)\,dy\,d\tau \\
&=\int_0^\infty \int_{\mathbb{R}^n}\frac1{(2 \pi \sinh 2\tau)^{n/2}}
\exp\big(-\tfrac14(|y|^2 \coth \tau+ |2x-y|^2 \tanh \tau)\big)f(t-\tau,x-y)\,dy\,d\tau \\
&=  \int_0^\infty \int_{\mathbb{R}^n}\frac1{(2 \pi \sinh 2\tau)^{n/2}}
\exp\big(-\tfrac14(|x-y|^2 \coth \tau+ |x+y|^2 \tanh \tau)\big)f(t-\tau,y)\,dy\,d\tau.
\end{aligned}
\end{equation}
We introduce the following notation
\begin{eqnarray*}
  S(\tau):= \frac1{(2 \pi \sinh 2\tau)^{n/2}}\, , &&\quad 
  H(\tau,y):= \exp(-\tfrac14 |y|^2 \coth \tau) \\
 \nonumber  G(\tau,x,y):=\exp (-\tfrac14 |2x-y|^2 \tanh \tau)\, ,&& \quad 
   F(\tau,x,y):=f(t-\tau,x-y).
\end{eqnarray*}
In particular, from \eqref{fundamental 2} we can write
$$u(t,x) = \int_0^\infty \int_{\mathbb{R}^n} S(\tau)H(\tau,y) G(\tau,x,y) F(\tau,x,y)\,dy\,d\tau.$$
In this proof we will use sometimes just $S,H,G,F$ in order to produce more readable formulas.
We shall need the following easy estimates.

\begin{rem}\label{seno}
Let $A$ be a positive constant so that $\tau < A$. There exists a constant $C_A$ such that 
  \begin{eqnarray*}
 \nonumber \sinh 2 \tau = \frac{e^{2\tau}- e^{-2\tau}}{2} =  e^{\theta} 4\tau
 \sim C_A \tau\, , && 
 \cosh 2 \tau =\frac{e^{2\tau}+ e^{-2\tau}}{2} \sim C_A, \\
 \coth 2 \tau \sim \frac{C_A}{\tau}\, , && \, \,\tanh \tau \sim C_A \tau.
\end{eqnarray*}
\end{rem}

\begin{lem}\label{lema2}
We have the following facts.
\begin{itemize}
\item[(i)] $\displaystyle \lim_{\varepsilon \rightarrow 0}\varepsilon^{-n/2}\int_{|y|< \varepsilon}e^{-|y|^2/(c\varepsilon)}\,dy=0$.
\item[(ii)]  $\displaystyle \int_0^{\varepsilon^2}\int_{|y|=\varepsilon} S(\tau) \partial_{y_i} H \frac{y_i}{|y|}\,d\sigma(y)\,d\tau \rightarrow -\frac{1}{n}\frac{1}{\Gamma(\frac{n}{2})}\int_\frac{1}{4}^\infty e^{-u}u^\frac{n}{2}\,\frac{du}{u}$, as $\varepsilon \rightarrow 0.$
 \item[(iii)]  $\displaystyle \int_{|y|<\varepsilon}S(\varepsilon^2)H(\varepsilon^2,y)\,dy \rightarrow
 \frac{1}{\Gamma(\frac{n}{2})}\int_0^\frac{1}{4}
 e^{-u}u^\frac{n}{2}\,\frac{du}{u}$, as $\varepsilon \rightarrow 0.$
\item[(iv)] $\displaystyle \int_0^{\varepsilon^2} \int_{|y|=\varepsilon} S H\,d\sigma(y)\,d\tau\rightarrow 0$, as $\varepsilon \rightarrow 0.$
\end{itemize}
  \end{lem}
  
\begin{proof}
For (i), use that
$$\varepsilon^{-n/2}\int_{|y|<\varepsilon} e^{-|y|^2/(c\varepsilon)}\,dy  = \int_{|z|< \sqrt{\varepsilon}}e^{-|z|^2/c}\,dz.$$
Observe that the double integral in (ii) equals to
\begin{align*}
-&\frac1{2n (2\pi)^{n/2}}  \int_0^{\varepsilon^2} \int_{|y|=\varepsilon}\frac{|y|\coth\tau
\exp(-\tfrac14|y|^2\coth \tau)}{( \sinh 2\tau)^{n/2}}\,d\sigma(y)\,d\tau\\
&=-\frac{1}{2^{n/2+1}n(2\pi)^{n/2}}\int_0^{\varepsilon^2}\int_{|y|=\varepsilon}
\frac{|y|\exp(-\frac{1}{4}|y|^2\frac{ \cosh\tau}{\sinh\tau})}{(\sinh\tau)^{\frac{n}{2}+1}(\cosh\tau)^{\frac{n}{2}-1}}
\,d\sigma(y)\,d\tau\\
&=-\frac{1}{2^{n/2+1}n(2\pi)^{n/2}}\int_0^{\varepsilon^2}\int_{|y|=\varepsilon}|y|
\bigg(\frac{\exp(-\frac{1}{4}|y|^2\frac{ \cosh\tau}{\sinh\tau})}{(\sinh\tau)^{\frac{n}{2}+1}(\cosh\tau)^{\frac{n}{2}-1}}
-\frac{\exp(-\frac{|y|^2}{4\tau})}{\tau^{\frac{n}{2}+1}} \bigg)\,d\sigma(y)\,d\tau\\
&\quad-\frac{1}{2^{n/2+1}n(2\pi)^{n/2}}\int_0^{\varepsilon^2} \int_{|y|=\varepsilon}
|y|\frac{\exp(-\frac{|y|^2}{4\tau})}{\tau^{\frac{n}{2}+1}}\,d\sigma(y)\,d\tau \\
&=:-\frac{1}{2^{n/2+1}n(2\pi)^{n/2}}( I_{\varepsilon 1}+ I_{\varepsilon 2}).
 \end{align*}
We write
$$I_{\varepsilon 1} = \int_0^{\varepsilon^2}\int_{|y|=\varepsilon}|y|
\big(M(\sinh \tau,\cosh \tau)-M(\tau,1)\big)\,d\sigma(y)\,d\tau,$$
with $\displaystyle M(u,v)=\frac{1}{u^{\frac{n}{2}+1} v^{\frac{n}{2}-1} }e^{-\frac{|y|^2v}{4u}}$.   
Now
$$\nabla M(u,v) =-\frac{1}{u^{\frac{n}{2}+1} v^{\frac{n}{2}-1} }e^{-\frac{|y|^2v}{4u}}
\bigg(\big(\tfrac{n}{2}+1\big)\frac{1}{u } -\frac{|y|^2v}{4u^2}, 
\big(\tfrac{n}{2} -1\big)\frac1{v}+ \frac{|y|^2}{4u}\bigg),$$
so that
$$|\nabla M(u,v)|\leq C\bigg(\frac{1}{u}+\frac{1}{v}\bigg)\frac{1}{u^{\frac{n}{2}+1}v^{\frac{n}{2}-1}}e^{-\frac{|y|^2v}{8u}}.$$
For  $\tau <1$ we have  
 $\sinh \tau- \tau \sim \cosh \tau-1\sim \tau^2$. Hence by the Mean Value Theorem we get 
\begin{align*}
| I_{\varepsilon 1}| &\le
C \int_0^{\varepsilon^2} \int_{|y|=\varepsilon} \tau^2\Big(\frac{|y| }{u} +\frac{|y|}{v}\Big)\frac{1}{u^{\frac{n}{2}+1} v^{\frac{n}{2}-1}} e^{-\frac{|y|^2}{8}\frac{v}{u}} d\sigma(y) d\tau\\
&\le C_n  \int_0^{\varepsilon^2}  \varepsilon^n \frac{1}{\tau^{\frac{n}{2}} } e^{-\frac{\varepsilon^2}{16}\frac{1}{\tau}} d\tau
= C\varepsilon^2\int_1^\infty u^{n/2-2}e^{-\frac{u}{16}}\,du\rightarrow 0,\quad\hbox{as}~\varepsilon\to0.
\end{align*}
On the other hand, integration using polar coordinates gives
$$I_{\varepsilon 2}=\int_0^{\varepsilon^2}\int_{|y|=\varepsilon}\frac1{\tau^{n/2}}\frac{|y|}{ \tau}
e^{-\frac{|y|^2}{4\tau}}\,d\sigma(y)\,d\tau
=\frac{4^\frac{n}{2}|S^{n-1}|}{\pi^\frac{n}{2}}\int_{1/4}^\infty u^\frac{n}{2} e^{-u}\frac{du}{u}.$$
For the proof of (iii) we will follow parallel ideas. Indeed, 
$$\int_{|y|<\varepsilon}S(\varepsilon^2) H(\varepsilon^2,y)\,dy
  = 2^{-\frac{n}{ 2}}\bigg( \int_{|y|<\varepsilon}\frac{ e^{-\frac{|y|^2}{4}\frac{1}{\sinh\varepsilon^2}\cosh\varepsilon^2 }}{(2\pi\sinh\varepsilon^2)^\frac{n}{2}(\cosh\varepsilon^2)^\frac{n}{2}} 
 \,dy \pm  \int_{|y|<\varepsilon}\frac{ e^{-\frac{|y|^2}{4}\frac{1}{\varepsilon^2} }}{(2\pi\varepsilon^2)^\frac{n}{2}}\,dy
 \bigg).$$
For the difference of the integrals we consider the function 
$\displaystyle N(u,v) = \frac{ e^{-\frac{|y|^2}{4}\frac{1}{u}v }}{(2\pi u)^\frac{n}{2}(v)^\frac{n}{2}} $. Then we have to estimate 
$$\int_{|y|<\varepsilon} \Big| N(\sinh \varepsilon^2,\cosh \varepsilon^2)-N(\varepsilon^2,1) \Big|\,d\sigma(y).$$
Following step by step the arguments in the proof of (i), we get that 
the difference of the integrals is bounded by 
$$C \int_{|y|< \varepsilon} \varepsilon^{-n+2} e^{-\frac{|y|^2}{8\varepsilon^2}} dy = C_n \varepsilon^2 \int_{|y| <1} e^{-\frac{|z|^2}{8}} dz \rightarrow 0.$$
Finally,
\begin{equation*}
\frac{1}{(2\pi)^\frac{n}{2}}\int_{|y|<\varepsilon}\frac{1}{(\varepsilon^2)^\frac{n}{2}} e^{-\frac{|y|^2}{4\varepsilon^2}} \,dy
    = \frac{|S^{n-1}|}{(2\pi)^\frac{n}{2}}\int_0^\varepsilon\frac{r^n}{\varepsilon^n} e^{-\frac{r^2}{4\varepsilon^2}} \,\frac{dr}{r}
    =\frac{2^n|S^{n-1}|}{2 (2\pi)^\frac{n}{2}}\int_0^\frac{1}{4}u^\frac{n}{2}e^{-u}\,\frac{ d u}{u}.
\end{equation*}
Summing up,  we get
 $$\lim_{\varepsilon \rightarrow 0} \int_{|y|<\varepsilon} S(\varepsilon^2)H(\varepsilon^2,y)\,dy  =  \frac{1}{\Gamma(\frac{n}{2})}\int_0^\frac{1}{4}u^\frac{n}{2}e^{-u}\,\frac{ d u}{u}. $$
For (iv) we just observe that 
$$\int_0^{\varepsilon^2} \int_{|y|=\varepsilon}S(\tau)H(\tau,y)\,dy\,d\tau
\leq  C_n \int_0^{\varepsilon^2} \frac1{\tau^{n/2}}\varepsilon^{n-1} e^{-\frac{\varepsilon^2}{c\tau}}\,d\tau
= C_n \varepsilon \int_{1/4} ^\infty u^{n/2-1} e^{-u}\, \frac{du}{u}\rightarrow0.$$
\end{proof}

Let us then continue with the proof of Theorem \ref{fundamental1} Part (A).
A rather parallel argument to the one we gave in Theorem \ref{fundamental 1}
gives the absolutely convergence of the integral in the statement of Theorem \ref{fundamental1}.
Since the product $G(x,y,\tau)F(x,y,\tau)$ is smooth with compact support,
then by dominated convergence,
$$\partial_{x_ix_i} u(t,x) =
\int_0^\infty \int_{\mathbb{R}^n} S(\tau)H(\tau,y) \partial_{x_ix_i}\big(G(\tau,x,y)F(\tau,x,y)\big)\,dy\,d\tau.$$
The following identities are easy to check
\begin{align*}
\partial_{x_i} G &= -2 \partial_{y_i} G, \qquad
\partial_{x_i}F = -\partial_{y_i} F, \qquad
\partial_{x_ix_i} G  = 4 \partial_{y_iy_i} G, \\
\partial_{x_ix_i}F&=\partial_{y_iy_i} F, \qquad
\partial_{x_ix_i}(GF) = 4 F\partial_{y_iy_i}G
+4\partial_{y_i} G\partial_{y_i} F +G \partial_{y_iy_i} F.
\end{align*}
By using the last list of formulas we have
$$\partial_{x_ix_i}u(t,x)=
\int_0^\infty \int_{\mathbb{R}^n} S(\tau)H(\tau,y)\Big[ 4 F\partial_{y_iy_i}G
+4 \partial_{y_i} G\partial_{y_i} F +G\partial_{y_iy_i} F\Big]\,dy\,d\tau.$$
A parallel argument as the one needed in order to prove the existence
of $u(t,x)$ shows that the last three  obvious integrals are  absolutely convergent.
Then, given  $\Omega_\varepsilon=\{ (t,y): \,\max(\tau^\frac{1}{2},|y|)>\varepsilon\}$, we can write 
\begin{equation}\label{id}
\begin{aligned}
\partial_{x_ix_i}u(t,x)&= \lim_{\varepsilon\rightarrow0}\int\int_{\Omega_\varepsilon}\big(4 S H F \partial_{y_iy_i}G
 + 4 S H \partial_{y_i} G\,\, \partial_{y_i} F + S H G\partial_{y_iy_i}F \big)\,dy\,d\tau\\
&=:\lim_{\varepsilon \rightarrow 0} \Big( I^\varepsilon_1+I^\varepsilon_2+I^\varepsilon_3 \Big).
\end{aligned}
\end{equation}
By integration by parts
$$I^\varepsilon_2 = -4\iint_{\Omega_\varepsilon} S(\tau)  \partial_{y_i}( H \partial_{y_i}G ) F \, dy\, d\tau
+  4\iint_{\partial\Omega_\varepsilon} S(\tau)H(\tau,y) (\partial_{y_i}G) \,F  \nu_i\,d\sigma(y,\tau),$$
where $\nu_i$ is the $i$th component of the outer unit normal vector of $\partial \Omega_\varepsilon$. 
As in the proof of Theorem \ref{fundamental 1} we decompose
$\partial\Omega_\varepsilon=\partial\Omega_\varepsilon^1+\partial\Omega_\varepsilon^2+\partial\Omega_\varepsilon^3$. 
Parallel to that case we have 
$$\iint_{\partial \Omega_\varepsilon^1} S(\tau)H(\tau,y) (\partial_{y_i}G) \,F
  \nu_i\,d\sigma(y,\tau) =0 = \iint_{\partial \Omega_\varepsilon^3} S(\tau)H(\tau,y) (\partial_{y_i}G) \,F  \nu_i\,d\sigma(y,\tau). $$
As $\displaystyle\partial_{y_i}G= \frac12(2x_i-y_i) \tanh\tau e^{-\frac{|2x-y|^2}{4}\tanh\tau},$
by using Remark \ref{seno} we get
\begin{multline}
\iint_{\partial\Omega_\varepsilon^2} S(\tau)H \partial_{y_i}G \,F \nu_i\,d\sigma(y,\tau)
   \leq C \int_0^{\varepsilon^2}\int_{|y|=\varepsilon} \frac{(\tanh\tau)^\frac12}{(\sinh\tau)^{\frac{n}{2}}}
         e^{-\frac{|y|^2}{4}\coth\tau} \,d\sigma(y)\,d\tau\\
   \label{estimate} \leq  C \int_0^{\varepsilon^2}\varepsilon^{n-1} \frac{\tau \tau^\frac12}{\tau^\frac{n}{2}}
           e^{-\frac{\varepsilon^2}{4\tau}}\,\frac{d\tau}{\tau}=C \varepsilon^2\int_1^\infty u^{\frac{n}{2}-1-\frac{1}{2}}e^{-u/4}\frac{du}{u}
        \rightarrow 0,~\hbox{as}~\varepsilon \rightarrow 0.
\end{multline}
Regarding $I_3^\varepsilon$, integration by parts gives
$$I^\varepsilon_3=-\iint_{\Omega_\varepsilon} S(\tau)\partial_{y_i}(H G) \partial_{y_i} F\,dy\, d\tau
+\iint_{\partial\Omega_\varepsilon}S(\tau)HG \partial_{y_i}F \nu_i\,dy\,d\tau,~\hbox{for}~i=1,\ldots,n.$$
By the same argument as we have used for $I^\varepsilon_2$, we get
$$\iint_{\partial\Omega_\varepsilon} S(\tau)H G \partial_{y_i} F \nu_i\,d\sigma(y, \tau)
\leq C \varepsilon \int_1^\infty u^{\frac{n}{2}-1}e^{-u/4} \,\frac{du}{u} \rightarrow 0,
~\hbox{as}~\varepsilon \rightarrow 0.$$
Hence
$$I^\varepsilon_3= -\iint_{\Omega_\varepsilon} S(\tau) \partial_{y_i}(H G) \partial_{y_i}F \,dy\,d\tau.$$
Again integration by parts gives 
$$I^\varepsilon_3= \iint_{\Omega_\varepsilon}S(\tau)\partial^2_{y_i^2}(H G) F \,dy\,d\tau
-\iint_{\partial\Omega_\varepsilon}S(\tau)\partial_{y_i}(H G) F \nu_i \,d\sigma(y)\,d\tau,~\hbox{for}~i=1,\ldots,n.$$
Parallel to $I^\varepsilon_2$ we have 
\begin{align*}
\iint_{\partial\Omega_\varepsilon}&S(\tau)\partial_{y_i}(H G) F \nu_i \,d\sigma(y,\tau)
    =\iint_{\partial\Omega_\varepsilon^2}S(\tau)\partial_{y_i}(H G) F \nu_i \,d\sigma(y,\tau)\\
&=-\int_0^{\varepsilon^2}\int_{|y|=\varepsilon} S(\tau)(\partial_{y_i}H) G F  \frac{y_i}{|y|}\,d\sigma(y)\,d\tau
  - \int_0^{\varepsilon^2}\int_{|y|=\varepsilon} S(\tau)H  (\partial_{y_i}G)  F  \frac{y_i}{|y|}\,d\sigma(y)\,d\tau\\
&=:I^\varepsilon_{31}+I^\varepsilon_{32}.
\end{align*}
By the same argument as in (\ref{estimate}) we get
$\lim_{\varepsilon \rightarrow 0} I^\varepsilon_{32} = 0. $
For $I^\varepsilon_{31}$ we proceed as follows:
\begin{align*}
I_{31}^\varepsilon&=- \int_0^{\varepsilon^2}\int_{|y|=\varepsilon}S(\tau) \partial_{y_i}H
\Big( G(\tau,x,y)F(\tau,x,y)- G(0,x,0)F(0,x,0)\Big)\frac{y_i}{|y|}\, d \sigma(y)\,d\tau \\
&\quad-\int_0^{\varepsilon^2}\int_{|y|=\varepsilon}S(\tau) 
\partial_{y_i}H G(0,x,0)F(0,x,0) \frac{y_i}{|y|}\, d \sigma(y)\,d\tau \\
&=: I^\varepsilon_{311}+I^{\varepsilon}_{312}.
\end{align*}
For $I^\varepsilon_{311}$, let $ \displaystyle g(\rho,x,z)= G(\rho,x,z)F(\rho,x,z)
=e^{-\frac{|2x-z|^2}{4}\tanh\tau}f(t-\rho,x-z)$. Since $\partial_{z_i}g$ and $\partial_\rho g $
 are smooth functions with compact support,
 $ \displaystyle \big|\nabla_{z,\rho}g(\rho,x,z)\big| \leq C$.  By Lemma \ref{lema2},
\begin{align*}
|I^\varepsilon_{311} | & \leq C\int_0^{\varepsilon^2}\int_{|y|=\varepsilon} S \partial_{y_i}H\cdot
(|y|+\tau) \frac{y_i}{|y|}\,d\sigma(y)\,d\tau
\leq C \varepsilon  \int_0^{\varepsilon^2}\int_{|y|=\varepsilon} S \partial_{y_i}H \frac{y_i}{|y|}\,d\sigma(y)\,d\tau,
\end{align*}
which vanishes as $\varepsilon \rightarrow 0$.
Now we compute exactly the integral $I^\varepsilon_{312}$. As $G(0,x,0)=1$, Lemma \ref{lema2} gives
$$I^\varepsilon_{312} = -f(t,x)\int_0^{\varepsilon^2}\int_{|y|=\varepsilon}S \partial_{y_i}H \frac{y_i}{|y|} \, d \sigma(y)\,d\tau
=f(t,x) \frac{1}{n\Gamma(\frac{n}{2})} \int_{1/4}^\infty e^{-u} u^\frac{n}{2} \,\frac{d u}{u}.$$
Using the formula
$$4S H \partial^2_{y_i}G - 4S \partial_{y_i} (H \partial_{y_i}G ) + S \partial^2_{y_i}(H G)
=S \partial_{y_iy_i}H G -2 S \partial_{y_i}H \partial_{y_i}G + S H \partial_{y_iy_i}G,$$
together with \eqref{id}, we get 
$$\partial_{x_ix_i}u(t,x)=\lim_{\varepsilon\rightarrow0} \int\int_{\Omega_\varepsilon}\Big(S\partial_{y_iy_i}HG-2S
\partial_{y_i}H \partial_{y_i}G + S H \partial_{y_iy_i}G\Big)f(t-\tau,x-y)\,dy\,d\tau-A_nf(t,x),$$
with $A_n $ as in \eqref{eq:AyB}. Observe  that in the case of
$\partial_{x_ix_j}u(t,x)$ some minor changes have to be done along the proof. For example,
$2S\partial_{y_i}H\partial_{y_i}G$ should be substituted  by $S\partial_{y_i}H\partial_{y_j}G+ S\partial_{y_j}H\partial_{y_i}G$.
On the other hand, the integral $I^\varepsilon_{312}$ is zero,
because of the presence of the term $\partial_{y_j}\big(\frac{y_i}{|y|}\big).$

Next, by following parallel arguments as before, we have 
$$\partial_t u(t,x)=-\lim_{\varepsilon \rightarrow 0}\iint_{\Omega_\varepsilon}S(y)H(\tau,y)
G(\tau,x,y)\partial_\tau f(t-\tau,x-y)\,dy\,d\tau.$$
Integration by parts gives
\begin{align*}
\iint_{\Omega_\varepsilon} S(y)H(\tau,y)&G(\tau,x,y)\partial_\tau f(t-\tau,x-y)\,dy\,d\tau \\
&=-\iint_{\Omega_\varepsilon}\partial_\tau( S(y)H(\tau,y) G(\tau,x,y)) f(t-\tau,x-y)\,dy\,d\tau\\
&\quad+\iint_{\partial\Omega_\varepsilon} S(y)H(\tau,y)G(\tau,x,y)f(t-\tau,x-y) \nu_\tau\,d\sigma(y,\tau).
\end{align*}
As a normal to $\partial\Omega_\varepsilon^2$ is given by  $\frac{1}{\varepsilon}(0,y)$ and,
on the other hand, $\lim_{\tau \rightarrow0}S(\tau)H(\tau,y)G(\tau,x,y)= 0$,  we have
$$\iint_{\partial\Omega_\varepsilon^2} SH Gf(t-\tau,x-y) \nu_\tau\,d\sigma(y,\tau)=0=
\iint_{\partial\Omega_\varepsilon^3} SH Gf(t-\tau,x-y)\nu_\tau\,d\sigma(y,\tau).$$
Finally, for the integral over $\partial\Omega_\varepsilon^1$,
\begin{align*}
&\iint_{\partial\Omega_\varepsilon^1}S(y)H(\tau,y) G(\tau,x,y)f(t-\tau,x-y) \nu_\tau\,d\sigma(y,\tau)\\
&=\iint_{\partial\Omega_\varepsilon^1}SH\big(G(\tau,x,y)F(t-\tau,x,y)-G(0,x,0)f(t,x)\big)d\sigma(y,\tau)+
\iint_{\partial\Omega_\varepsilon^1}S H G(0,x,0)f(t,x)d\sigma(y,\tau)\\
&=:I_{t1}+I_{t2}.
\end{align*}
Consider $\displaystyle g(\rho,x,z) = G(\rho,x,z)f(t-\tau,x-z)$. By the Mean Value Theorem we get
\begin{align*}
|I_{t1}|&=\bigg|\int_{|y| < \varepsilon}S(\varepsilon^2)H(y,\varepsilon^2)\big(g(\varepsilon^2,x,y)- g(0,x,0)\big) \,dy\bigg|
\le C \int_{|y| < \varepsilon} S(\varepsilon^2)H(\varepsilon^2,y) (|y|+\varepsilon^2) \,dy\\
&\le C \int_{|y| < \varepsilon} S(\varepsilon^2)H(\varepsilon^2,y) \varepsilon\,dy 
 \le C \varepsilon \int_0^1 u^n e^{-u}\,\frac{d u}{u } \rightarrow 0,~\hbox{as}~\varepsilon \rightarrow 0.
\end{align*}
Finally,  by using Lemma \ref{lema2} we get
$$I_{t2}= -f(t,x) \int_{|y| < \varepsilon} S(\varepsilon^2)H(\varepsilon^2,y) \,dy  = -f(t,x)\frac{1}{\Gamma( \frac{n}{2})}
\int_0^\frac{1}{4}e^{-u}u^\frac{n}{2}\,\frac{du}{u}.$$
In particular,
$$\partial_tu(t,x) = \lim_{\varepsilon \rightarrow 0}\iint_{\Omega_\varepsilon}
\partial_\tau\big(S(\tau)H(\tau,y)G(\tau,x,y)\big)f(t-\tau,x-y)\,dy\,d\tau+B_n f(t,x).$$
This finishes the proof of part A.

\subsection{Proof of Theorem \ref{fundamental1} Part (B)}

Now we  shall prove the weighted
$L^p$ results for the parabolic Hermite--Riesz transforms defined by the formulas \eqref{equation x}
and \eqref{equation t}. We shall denote them as
$$R_{ij}^{H}= \partial_{x_ix_j}(\partial_t+H)^{-1}\quad\hbox{and}\quad R_t^{H}=\partial_t(\partial_t+H)^{-1}.$$

As we saw in Definition \ref{CZ}, in order to apply the general theory of Calder\'on--Zygmund operators
we need the boundedness of the operator in some $L^p$ space.
We prove that these Riesz transforms are bounded in $L^2$. In our opinion this result is of independent interest.

\begin{thm}\label{L2} 
The parabolic Hermite--Riesz transforms
$R^{{H}}_{ij}$ and $R_t^{{H}}$ are bounded operators in $L^2(\mathbb{R}^{n+1}).$
\end{thm}

\begin{proof} We shall use the basis given by the Hermite functions, $h_\alpha.$ 
Consider the collection of functions
$$\mathcal{A} =\bigg\{\sum_{\alpha \in \mathcal{J}} \varphi_\alpha(t)h_\alpha(x):\mathcal{J}\subset\N_0^n~\hbox{(finite)},
~\widehat{\varphi_\alpha}\in C^\infty_c(\R)\bigg\}.$$
By using \eqref{FHermite},
\begin{equation*}
\begin{aligned}
(\partial_t+H)^{-1}\bigg[\sum_{\alpha \in \mathcal{J}} \varphi_\alpha(t)h_\alpha(x)\bigg] 
&= \int_0^\infty \sum_{\alpha \in \mathcal{J}} e^{-\tau \partial_t}\varphi_\alpha(t) e^{-\tau H} h_\alpha(x)\, d\tau\\
&= \sum_{\alpha \in \mathcal{J}} \int_0^\infty \varphi_\alpha(t-\tau) e^{-\tau(2|\alpha|+n)}h_\alpha(x)\,d\tau.
\end{aligned}
\end{equation*}
For the operator $R_t^{{H}}$ we have
\begin{align*}
\Big\|\partial_t (\partial_t+H_x)^{-1}& \sum_{\alpha \in \mathcal{J}} \varphi_\alpha(t)h_\alpha(x) \Big\|^2_{L^2(\R^{n+1})} =\Big\|  \sum_{\alpha \in \mathcal{J}} \xi \widehat{\varphi_\alpha}(\xi)(i\xi+2|\alpha|+n)^{-1}h_\alpha(x) \Big\|^2_{L^2(\R^{n+1})}\\ 
&\le C   \sum_{\alpha \in \mathcal{J}} \|\widehat{\varphi_\alpha}(\xi)\|^2_{L^2(\R^{n+1})}
=  C\Big\|\sum_{\alpha \in \mathcal{J}} \varphi_\alpha(t)h_\alpha(x)\Big\|^2_{L^2(\R^{n+1})}.
\end{align*}
It is well known that the second order Hermite-Riesz operators type, we mean all kind of combinations $(\partial_{x_i}\pm x_i)(\partial_{x_j}\pm x_j) H^{-1}$, are bounded in $L^2(\mathbb{R}^n)$. Hence in order to prove  $R^{H}_{ij}$ are bounded in $L^2(\mathbb{R}^{n+1}) $ it is enough to prove that $H  (\partial_t+H_x)^{-1}$ is bounded.  For that purpose we have
\begin{align*}
\Big\|H(\partial_t+H_x)^{-1}& \sum_{\alpha \in \mathcal{J}} \varphi_\alpha(t)h_\alpha(x) \Big\|^2_{L^2(\R^{n+1})} =\Big\|  \sum_{\alpha \in \mathcal{J}} \widehat{\varphi_\alpha}(\xi)(2|\alpha|+n)(i\xi+2|\alpha|+n)^{-1}h_\alpha(x) \Big\|^2_{L^2(\R^{n+1})}\\ 
&\le C   \sum_{\alpha \in \mathcal{J}} \|\widehat{\varphi_\alpha}(\xi)\|^2_{L^2(\R^{n+1})}
=  C\Big\|\sum_{\alpha \in \mathcal{J}} \varphi_\alpha(t)h_\alpha(x)\Big\|^2_{L^2(\R^{n+1})}.
\end{align*}
\end{proof}

To compute the kernels, we perform the change of variables $x-y\longmapsto y$ in \eqref{equation x} and \eqref{equation t} to get 
\begin{multline*}
R_{ij}^{H}f(t,x)\\=\lim_{\varepsilon\rightarrow0}\iint_{\Omega_\varepsilon(x)}\Big(S\partial_{y_iy_j}HG
-S \partial_{y_i}H \partial_{y_j}G -S \partial_{y_j}H \partial_{y_i}G+ S H \partial_{y_iy_j}G\Big) f( t-\tau,y )\,dy\,d\tau-A_nf(t,x),
\end{multline*}
 and 
$$R_t^{H}f(t,x)=\lim_{\varepsilon \rightarrow 0}\iint_{\Omega_\varepsilon(x)}\partial_\tau(SHG)f(t-\tau,y)\,dy\,d\tau+B_n f(t,x).$$
The functions appearing in the integrand are $S(\tau)$ as before and
$$H(\tau,x-y)= \exp\big(-\tfrac14|x-y|^2\coth\tau\big),\quad
G(\tau,x,x-y) = \exp\big(-\tfrac14|x+y|^2 \tanh \tau\big).$$
Let us see that the kernels of the operators above satisfy the standard Calder\'{o}n--Zygmund estimates.
On the way we need some easy estimates that we present here for future reference.
 
\begin{rem}\label{seno 2}
Let $\tau >0$.
\begin{itemize}
\item[(a)] If $\tau < 1$ then $\sinh\tau\sim \tau,~\cosh\tau\sim C,~\coth\tau\sim \frac{1}{\tau},~\hbox{and}~\tanh\tau\sim\tau$.
\item[(b)] If $\tau >1$ then $\sinh\tau \sim e^\tau,~\cosh\tau\sim e^\tau,~\coth\tau\sim C,~\hbox{and}~\tanh\tau \sim C$.
\end{itemize}
\end{rem}

It is easy to check that
\begin{equation}\label{compute H}
\begin{aligned}
|\partial_{y_i}H| &\leq  C  (\coth\tau)^\frac{1}{2} e^{-\frac{|x-y|^2}{4}\coth\tau}, \quad \quad 
|\partial_{y_iy_j}H | \leq  C \coth\tau e^{-\frac{|x-y|^2}{4}\coth\tau},\\
|\partial_{y_iy_jy_k}H| &\leq C  (\coth\tau)^\frac{3}{2} e^{-\frac{|x-y|^2}{4}\coth\tau}.
\end{aligned}
\end{equation}
Also 
\begin{equation}\label{compute G}
\begin{aligned}
|\partial_{y_i}G| &\leq C  (\tanh\tau)^\frac{1}{2} e^{-\frac{|x+y|^2}{4}\tanh\tau},\quad\quad 
|\partial_{y_iy_j}G | \leq  C \tanh\tau e^{-\frac{|x+y|^2}{4}\tanh\tau}.\\
|\partial_{y_iy_jy_k}G | &\leq C  (\tanh\tau)^\frac{3}{2} e^{-\frac{|x+y|^2}{4}\tanh\tau}.
\end{aligned}
\end{equation}
If $\tau <1$ then by using \eqref{compute H}, \eqref{compute G} and Remark \ref{seno 2},
\begin{equation}\label{size 11}
|S\partial_{y_iy_j} HG| \le \frac{C}{\tau^{\frac{n}{2}+1}} e^{-\frac{|x-y|^2}{c\tau}}
 \leq \frac{C}{(|x-y|+\tau^{1/2})^{n+2}}.
 \end{equation}
On the other hand, when $\tau>1$,
\begin{equation}\label{size 12}
|S\partial_{y_iy_j} HG| \le Ce^{-C(\tau+|x-y|^2)}\leq \frac{C}{(|x-y|+\tau^{1/2})^{n+2}}.
\end{equation}
Analogously,
\begin{equation}\label{size 2}
|S \partial_{y_i}H \partial_{y_j}G| \leq \frac{C}{(\sinh{2\tau})^\frac{n}{2}} e^{-\frac{|x-y|^2}{4}\coth\tau}.
\end{equation}
For  $\tau < 1$,  as  $\frac{1}{\tau}>1$,  
$$|S \partial_{y_i}H \partial_{y_j}G|\leq \frac{C}{\tau^\frac{n}{2}} e^{-\frac{|x-y|^2}{4\tau}}
   \leq\frac{C}{\tau^{\frac{n}{2}+1}} e^{-\frac{|x-y|^2}{4\tau}}.$$
If $\tau > 1$,  we have 
$$|S \partial_{y_i}H \partial_{y_j}G| \leq C e^{-n\tau}e^{-\frac{|x-y|^2}{C}}.$$
By a parallel  argument to the one in the proof of \eqref{size 11} and \eqref{size 12}
we get that \eqref{size 2} is controlled by
$$\frac{C}{(|x-y|+\tau^\frac12)^{n+2}}.$$
By the estimates \eqref{compute G} and parallel arguments to the previous cases we get 
\begin{equation}\label{size 3}
| S H\partial_{y_iy_j}G | \leq C \frac{\tanh \tau}{(\sinh{2\tau})^\frac{n}{2}} e^{-\frac{|x-y|^2}{4}\coth\tau} 
\le\frac{C}{(|x-y|+\tau^\frac12)^{n+2}}.
\end{equation}
Pasting together  \eqref{size 12}, \eqref{size 2} and \eqref{size 3}  we get
$$|R^H_{ij}(\tau,x,y)|\le\frac{C}{(|x-y|+\tau^\frac12)^{n+2}}.$$
Analogously, it can be proved that 
$$|\nabla_y R^H_{ij}(\tau,x,y)|+|\nabla_xR^H_{ii,x}(\tau,x,y)|\le\frac{C}{(|x-y|+\tau^\frac12)^{n+3}},$$
and 
$$|\partial_\tau R^H_{ij} (\tau,x,y) | \le \frac{C}{(|x-y|+\tau^{1/2})^{n+4}}.$$
Observe that if $4(|s|^{1/2}+ |y-y_0|)  \le |\tau-s|^{1/2} +|x-y_0|,$ then 
$|\tau-s|^{1/2} +|x-y_0|\sim |\tau|^{1/2} +|x-y_0|$ and any intermediate point 
$(\tau-\theta s,x,y_\theta)$ between $(\tau,x,y)$ and $(\tau-s,x,y_0)$ satisfies
$|\tau- \theta s|^{1/2} + |x-y_\theta| \le C(|\tau-  s|^{1/2} + |x-y_0|)$. Hence
\begin{align*}
|R^H_{ij}(\tau,x,y)- R^H_{ij}(\tau-s,x,y_0)| &\le | \partial_\tau R^H_{ij}(\tau-\theta s,x,y_\theta)|\, |s|+|\nabla_yR^H_{ij}(\tau-\theta s,x,y_\theta)| |y-y_0|\\
&\le \frac{C |s|}{(|\tau-\theta s|^{1/2}+ |x-y_\theta|)^{n+4}} + \frac{C |y-y_0|}{(|\tau-\theta s|^{1/2}+ |x-y_\theta|)^{n+3}} \\
&\le  \frac{C |s|^{1/2} ( |\tau-s|^{1/2}+ |x-y_0|)}{(|\tau-s|^{1/2}+ |x-y_0|)^{n+4}} + \frac{C |y-y_0|}{(| \tau-s|^{1/2}+ |x-y_0|)^{n+3}} \\
 &\le C\frac{|s|^{1/2}+  |y-y_0|}{(| \tau-s|^{1/2}+ |x-y_0|)^{n+3}}.
\end{align*} 
In other words the kernel $R^H_{ij}(\tau,x,y)$ satisfies the smoothness condition (II.2) of Definition \ref{CZ}.

Finally, we can apply the full strength of the theory of Calder\'on--Zygmund operators
that we presented at the beginning of Section \ref{pre},
arriving to the end of the proof of Theorem \ref{fundamental1} Part (B).

\subsection{Proof of Theorem \ref{fundamental3}.}

We just give a sketch of the proof.
As in the case of the heat equation, see Section \ref{pre}, it is enough to study the problem
$$\begin{cases}
\partial_tv=\Delta v-|x|^2v+f,&\hbox{for}~t >0,~x\in \mathbb{R}^n,\\
v(0,x)=0,&\hbox{for}~x\in\mathbb{R}^n.
\end{cases}$$
By a parallel argument to the one in the proof of Theorem \ref{fundamental 3} we get a classical solution 
for functions $f$ that are $C^2$ and have compact support.
The weighted $L^p$ boundedness and the weighted
weak (1,1) estimate can be obtained by comparison with the case of the whole space.
The details are left to the interested reader.

\section{Weighted mixed-norm estimates: proof of Theorems \ref{mixriesz} and \ref{Poisson}}\label{mix}

Let us begin with the proof of Theorem \ref{Poisson}. We will show
how the ideas work for the Poisson operators first and then we will
show Theorem \ref{mixriesz}.

\begin{rem}\label{poisson}
We want to observe that formulas \eqref{eq:Poisson operator Laplaciano}
and \eqref{eq:Poisson operator Hermite} are in principle well defined for $u$ in $L^2(\R^{n+1})$.
To this end, we first notice that the following integral
$$\frac{y^{2s}}{4^s\Gamma(s)}\int_0^\infty e^{-y^2/(4\tau)}e^{-\tau(i\rho+\lambda)}\,\frac{d\tau}{\tau^{1+s}},
\quad\rho\in\R,~\lambda\ge0,~0<s<1.$$
is convergent, as the Cauchy Integral Theorem and analytic continuation of the formula with $\rho=0$ show. 
Notice that these integrals are related to Bessel functions, see \cite{Lebedev}.
Then \eqref{eq:Poisson operator Laplaciano} and \eqref{eq:Poisson operator Hermite} are well defined
by using Fourier transform and Hermite expansions, respectively. 
\end{rem}

\subsection{Proof of Theorem \ref{Poisson} when $q=p$}

We start with the case $q=p$.

\begin{rem}\label{product}
If a weight  $\nu=\nu(t)$ belongs to the class $A_p(\R)$, $1\leq p<\infty$,
it also satisfies the definition of Muckenhoupt condition by changing the distance $|t-s|$
by the distance $|t-s|^{1/2}$. As a consequence, if $\nu=\nu(t)\in A_p(\R)$
and $\omega=\omega(x)\in A_p(\R^n)$ then
the tensor product weight $w(t,x):=\nu(t)\omega(x)\in A_p^\ast(\R^{n+1})$, for any $1\leq p<\infty$.
\end{rem}

From Remark \ref{product} we notice that
our next result is slightly more general than Theorem \ref{Poisson} when $p=q>1$.

\begin{thm}\label{Poisson1}
Let $u\in L^{p}(\R^{n+1},w)$, for $w\in A_p^*(\R^{n+1})$, $1\leq p\leq\infty$. Let $\mathcal{P}^{s,\ast}u(t,x)$
be as in Theorem \ref{Poisson}. If $1<p\leq\infty$ then
$$\|\mathcal{P}^{s,\ast}u\|_{L^p(\mathbb{R}^{n+1},w)}\le
C_{n,p,w}\|u\|_{L^p(\mathbb{R}^{n+1},w)}.$$
If $p=1$ then, for every $\lambda>0$,
$$w\big(\{(t,x)\in\R^{n+1}:|\mathcal{P}^{s,\ast}u(t,x)|>\lambda\}\big)
\le\frac{C_{n,w}}{\lambda}\|u\|_{L^1(\R^{n+1},w)}.$$
\end{thm}

\begin{proof} 
We first show the computation for the case of the heat equation
\eqref{eq:Poisson operator Laplaciano}, that is, when $\mathcal{P}^{s,\ast}u(t,x)=\sup_{y>0}|P^s_{\Delta,y}u(t,x)|$.
For functions  $u\in L^p(\mathbb{R}^{n+1})$ we have 
$$P_{\Delta,y}^su(t,x)
= \frac{y^{2s}}{4^s\Gamma(s)}\int_0^\infty e^{-y^2/(4\tau)}\int_{\mathbb{R}^n} 
\frac{ e^{- \frac{|z|^2}{4\tau}}}{(4\pi\tau)^{n/2}} u(t-\tau,x-z)\,dz\,\frac{d\tau}{\tau^{1+s}}.$$
The operator above can be regarded as the convolution in $\mathbb{R}^{n+1}$ of $u$ with the $L^1(\mathbb{R}^{n+1})$ function 
\begin{equation}\label{2}
P^s_{\Delta,y}(t,x)=\frac{y^{2s}}{4^s\Gamma(s)}\frac{e^{-y^2/(4t)}}{t^{1+s}}  \frac{ e^{- \frac{|x|^2}{4t}}}{(4\pi t)^{n/2}}\chi_{t>0}.  
\end{equation}
Hence, for each $y>0$, the operator $P_{\Delta,y}^s$ maps $L^p(\mathbb{R}^{n+1}) $ into itself.
In order to show the boundedness of the maximal operator
we shall apply the theory of vector-valued Calder\'on--Zygmund operators.
We begin by observing that
\begin{equation}\label{1}
|P_{\Delta,y}^su(t,x) |\le\|u\|_{L^\infty(\R^{n+1})}\frac{y^{2s}}{4^s\Gamma(s)}\int_0^\infty
e^{-y^2/(4\tau)}\,\frac{d\tau}{\tau^{1+s}}= \| u\|_{L^\infty(\R^{n+1})}.
\end{equation} 
Consider then the $L^\infty(0,\infty)$-valued operator given by
\begin{equation}\label{PES}
\mathrm{P}_{\Delta}^su(t,x):=\Big\{P_{\Delta,y}^su(t,x)\Big\}_{y>0}.
\end{equation}
Then \eqref{1} shows that $\mathrm{P}^s_\Delta$ is a bounded operator from $L^\infty(\mathbb{R}^{n+1})$
into  $L^\infty(\R^{n+1};L^\infty(0,\infty))$.
Moreover it has a representation as a convolution with a vector-valued kernel
$$\mathrm{P}_\Delta^su(t,x)=\int_{\R^{n+1}}\big\{P_{\Delta,y}^s(\tau,z)\big\}_yu(t-\tau,x-z)\,dz\,d\tau,$$
where $P_{\Delta,y}^s(t,x)$ is given in \eqref{2}.
It is easy to check that
$$|P^s_{\Delta,y}(t,x)|\le \frac{C}{ (|t|^{1/2} +|x|)^{n+2}},$$
and
$$|\partial_tP^s_{\Delta,y}(t,x)|\le \frac{C}{(|t|^{1/2}+|x|)^{n+4}},
\quad\hbox{and}\quad|\nabla P^s_{\Delta,y}(t,x)| \le \frac{C}{(|t|^{1/2}+|x|)^{n+3}}.$$
Therefore $\big\{P^s_{\Delta,y}(t,x)\big\}_{y>0}$ is a vector-valued
Calder\'on--Zygmund kernel as described in Definition \ref{CZ}.
Thus the operator $\mathrm{P}_\Delta^s$ is bounded from $L^p(\mathbb{R}^{n+1},w)$ 
into  $L^p(\mathbb{R}^{n+1},w;L^\infty(0,\infty))$ for $w\in A_p^*(\R^{n+1})$
and it satisfies the corresponding weak-type estimate,
namely, if $p=1$ then, for any $\lambda>0$,
$$w\big(\{(t,x)\in\R^{n+1}:\|\mathrm{P}_\Delta^su(t,x)\|_{L^\infty(0,\infty)}>\lambda\}\big)\leq
\frac{C_n}{\lambda}\|u\|_{L^1(\R^{n+1},w)}, \quad\hbox{for}~w\in A_1^*(\R^{n+1}).$$
The statement then follows by observing that
$\mathcal{P}^{s,*}_\Delta u(t,x)=\|\mathrm{P}^s_\Delta u(t,x)\|_{L^\infty(0,\infty)}$.

Next let us consider the case of the harmonic oscillator evolution equation \eqref{eq:Poisson operator Hermite}, namely,
when $\mathcal{P}^{s,\ast}u(t,x)=\sup_{y>0}|P^s_{H,y}u(t,x)|$.
By using Remark \ref{seno 2}, see also \cite{Stempak-Torrea,Thangavelu},
it is easy to check that the Mehler kernel $\mathcal{W}_t(x,z)$
given in (\ref{Hheat}) and (\ref{fundamental 2}) satisfies 
$$\mathcal{W}_t(x,y)\le C\frac{e^{-|x-y|^2/(ct)}}{t^{n/2}},$$
where $C$ and $c$ are positive constants. In a similar way, its derivatives can be estimated
by the same bounds for the Gauss--Weierstrass kernel. Then we can proceed exactly as in the
proof for the case of $P^s_{\Delta,y}$ and conclude.
\end{proof}

\subsection{Proof of Theorem \ref{Poisson}}

The kernels of both operators $P^s_{\Delta,y}$ and $P^s_{H,y}$ are bounded by
$$K^s_{y}(t,x)=Cy^{2s}\frac{e^{-y^2/(ct)}}{t^{1+s}} \frac{e^{-|x|^2/(ct)}}{t^{n/2}}\chi_{t>0}.$$
Moreover 
\begin{equation}\label{operatornorm}
\int_{\R^n}K^s_y(t,x)\,dx \le Cy^{2s} \frac{e^{-y^2/(ct)}}{t^{1+s}}\chi_{t>0}.
\end{equation}
We show the case of the heat
equation, the other one follows the same lines.
Let us fix $1<p<\infty$.
By using Theorem \ref{Poisson1} and Remark \ref{product} we have
$$\|\mathcal{P}^{s,*}u\|_{L^p(\R,\nu;L^p(\R^n),\omega)}\le C\|u\|_{L^p(\R,\nu;L^p(\R^n,\omega))}.$$
This estimate with $\nu=1$ implies in particular the boundedness
$$\mathrm{P}^s_\Delta: L^p(\mathbb{R};L^p(\mathbb{R}^n,\omega))\longrightarrow L^p(\mathbb{R};
(L^p(\mathbb{R}^n,\omega;L^\infty(0,\infty)))),$$
where $\mathrm{P}^s_\Delta$ is given in the proof of Theorem \ref{Poisson1} by \eqref{PES}.
The kernel of $\mathrm{P}^s_\Delta$ is given by
\begin{eqnarray} \label{CZP}
\{P_{\Delta,y}^s(t,\cdot) \}_y&:& L^p(\mathbb{R}^n,\omega) \longrightarrow L^p(\mathbb{R}^n,\omega;L^\infty(0,\infty))   \\
   & & \quad  \varphi \quad  \,\,\, \longrightarrow \{ P_{\Delta,y}^s(t,\cdot)\ast\varphi(x)\}_y. \nonumber
\end{eqnarray}
In the case  $\omega =1$, Young's inequality and \eqref{operatornorm} guarantee that the norm of the operator above
is bounded by $Ct^{-1}$. Also in this case it is easy to see that the kernel 
$ \{\partial_tP_{\Delta,y}^s(t,\cdot) \}_y$ has norm 
bounded by $Ct^{-2}$.  On the other hand, for each fixed $t$, the $L^\infty(0,\infty)$-norm of the kernel \eqref{CZP}
is bounded by $Ct^{-1}\frac{e^{-|x|^2/(ct)}}{t^{n/2}} \chi_{t>0}\le Ct^{-1}/|x|^n$, while its gradient 
with respect to $x$ is bounded by $Ct^{-1}/|x|^{n+1}$. Then the Calder\'on--Zygmund theory
gives that the operator norm of \eqref{CZP} is bounded by $C_1t^{-1}$ and,
by a parallel argument, the operator norm of $ \{\partial_tP_{\Delta,y}^s(t,\cdot) \}_y$ is bounded by
$C_2t^{-2}$, where the constants $C_1$ and $C_2$ depend on $\omega$.
These estimates ensure that the kernel satisfies the
standard estimates of Calder\'on--Zygmund kernels on the real line. Therefore we get the boundedness
$$\mathrm{P}_\Delta^s:L^q(\mathbb{R},\nu;L^p(\mathbb{R}^n,\omega))\rightarrow
L^q(\mathbb{R},\nu;L^p(\mathbb{R}^n,\omega;L^\infty(0,\infty))),\quad1<q<\infty,$$
and the corresponding weak type estimate when $q=1$.
The relation $\mathcal{P}^{s,*}u(t,x)=\|\mathrm{P}^s_\Delta u(t,x)\|_{L^\infty(0,\infty)}$
concludes the proof.

\subsection{Proof of Theorem \ref{mixriesz}}

We have already proved that each of the operators $\mathrm{R}=R^\Delta_{ij},R^\Delta_t,R^H_{ij},R^H_t$
are bounded in $L^p(\mathbb{R}^{n+1},w)$, $1<p<\infty$, $w\in A_p^\ast(\R^{n+1})$,
and satisfy the corresponding weak-type estimate when $p=1$ and $w\in A_1^\ast(\R^{n+1})$. Hence, for any $1<p<\infty$,  
$$\mathrm{R}:L^p(\R,\nu;L^p(\R^n,\omega))\longrightarrow
L^p(\R,\nu;L^p(\R^n,\omega)),\quad\hbox{for}~\nu\in A_p(\R),~\omega \in A_p(\R^{n}),$$ 
see Remark \ref{product}.
By following carefully the computations we made in Sections \ref{pre} and \ref{pf1}, it can be readily seen
that the kernels of these operators are bounded by $Ct^{-n/2-1}e^{-|x-y|^2/(ct)}$.
Therefore the arguments presented above for the Poisson kernels
remain valid in these cases and we get the mixed norm estimates.
Further details are left to the interested reader.

\medskip

\noindent\textbf{Acknowledgements.} We are grateful to Prof. Sundaram Thangavelu for 
enlightening conversations about the results of this paper.
In particular, the proof of Theorem \ref{L2} presented here was suggested by him.

The first author was supported by grants 11471251 and 11271293 from National
Natural Science Foundation of China. The second and third authors 
were supported by grant MTM2015-66157-C2-1-P from Government of Spain.
The results in this work are part of the first author's PhD Thesis.



\end{document}